\documentclass{amsart}

\usepackage[applemac]{inputenc}
\usepackage{amsmath}
\usepackage{bbm}
\usepackage{latexsym}
\usepackage{xcolor}
\usepackage{tikz}
\usetikzlibrary{arrows}
\usetikzlibrary{patterns}
\usetikzlibrary{shapes}
\usepackage{mathrsfs}
\usepackage{mathtools}
\usepackage{verbatim}
\usepackage{bbm}
\usepackage{hyperref}
\hypersetup{
	colorlinks,
	linkcolor={red!50!black},
	citecolor={blue!50!black},
	urlcolor={blue!80!black}
}
\usepackage[mathcal]{euscript}

\newcommand{\R}{\mathbb{R}}
\newcommand{\N}{\mathbb{N}}

\newcommand{\eexp}{\overrightarrow{\exp}}

\newcommand{\Vc}{\mathrm{Vec}}

\newcommand{\mfk}{\mathfrak}
\newcommand{\IM}{\mathrm{Im} \,}

\newcommand{\rank}{\mathrm{rank}}

\newcommand{\trace}{\mathrm{tr}}

\newcommand{\G}{\mathbb{G}}
\newcommand{\abn}[1]{\mathrm{Abn}_{\G}^{#1}}

\usepackage[usenames,dvipsnames]{pstricks}
\usepackage{epsfig}
\usepackage{pst-grad} % For gradients
\usepackage{pst-plot} % For axes
\DeclareMathOperator*{\Span}{span}
\newtheorem{thm}{Theorem}
\newtheorem{lemma}[thm]{Lemma}

\newtheorem{prop}[thm]{Proposition}

\theoremstyle{definition}

\newtheorem{defi}[thm]{Definition}
\newtheorem{remark}[thm]{Remark}

\newcommand{\be}{\begin{equation}}
\newcommand{\ee}{\end{equation}}

\usepackage{mathtools} %scommentare per far numerare solo le equationi che sono chiamate nel testo
\mathtoolsset{showonlyrefs,showmanualtags} %scommentare per far numerare solo le equationi che sono chiamate nel testo

\numberwithin{equation}{section}

\title{A dynamical approach to the Sard problem in Carnot groups}

\author[Boarotto]{Francesco Boarotto}
\address[Boarotto and Vittone]{Dipartimento di Matematica Tullio Levi-Civita,
	Universit\`a degli studi di Padova, Italy}
\email{\href{mailto:francesco.boarotto@math.unipd.it}{\nolinkurl{francesco.boarotto@math.unipd.it}}}

\author[Vittone]{Davide Vittone}
\email{\href{mailto:vittone@math.unipd.it}{\nolinkurl{vittone@math.unipd.it}}}

\date{\today}

\thanks{The authors are supported by the University of Padova STARS Project ``Sub-Riemannian Geometry and Geometric Measure Theory Issues: Old and New'' (SUGGESTION), and by GNAMPA of INdAM (Italy) through projects ``Rectifiability in Carnot Groups'' and ``Applicazioni della Teoria delle Correnti all'Analisi Reale e al Trasporto Ottimo''.}


\begin{document}
	\begin{abstract}
	We introduce a dynamical-systems approach for the study of the Sard problem in sub-Riemannian Carnot groups. We show that singular curves can be obtained by concatenating trajectories of suitable dynamical systems. As an applications, we positively answer the Sard problem in some classes of Carnot groups.
	\end{abstract}
	
	\maketitle
	
	\section{Introduction}\label{sec:intro}
	
It can be safely stated that, despite the explosion of interest it has witnessed in the last decades,  plenty of  questions  pertaining to sub-Riemannian geometry remain elusive even among the  foundational ones. One of them is surely the so-called  Sard problem, that is presently unsolved  even in rich structures such as  Carnot groups. In this paper we intend to give a contribution to this problem, as we now explain.

		Remember that a {\em Carnot group} $\G$ of rank $r$ and step $s$ is a connected, simply connected and nilpotent Lie group whose Lie algebra $\mfk{g}$, here identified with the tangent at the group identity $e$, admits a stratification of the form:
	\be\label{eq:stratification}
		 \mfk{g}=\mfk{g}_1\oplus \dots \oplus \mfk{g}_s,
	\ee 
	with $\mfk{g}_{i+1}=[\mfk{g},\mfk{g}_i]$ for $1\le i\le s-1$, $[\mfk g,\mfk g_s]=\{0\}$ and $\dim(\mfk{g}_1)=r$. A Carnot group can be naturally endowed with a sub-Riemannian structure by declaring the first layer $\mfk g_1$ of the Lie algebra to be the horizontal space. Actually, Carnot groups are infinitesimal models for sub-Riemannian manifolds (that we do not introduce here, see \cite{ABB,JeanLibro,Montgomery,RiffordLibro}). Denoting by $L_g$ the left-translation on $\G$ by an element $g\in\G$, we consider the {\em endpoint map}
	\be\label{eq:endpoint_intro}
		\begin{aligned}
			F_e: L^1([0,1],\mfk{g}_1)&\to \G,\\
				u&\mapsto \gamma_u(1),
		\end{aligned}
	\ee
	where we denoted by $\gamma_u:[0,1]\to \G$ the absolutely continuous curve issuing from  $e$, whose derivative is given by $(dL_{\gamma_u(t)})_eu(t)$ for a.e. $t\in [0,1]$. Any such curve $\gamma_u$ is called {\em horizontal}.
	
	\begin{defi}\label{defi:abnormal_set}
		Given a Carnot group $\G$, we denote by $\abn{}\subset \G$ the set of the {\em singular values} of $F_e$. In particular, a point $g\in \G$ belongs to $\abn{}$ if and only if there exists a horizontal curve $\gamma_u$ joining $e$ and $g$ that is associated with a critical value $u$ of the differential $dF_e$. 
		
		\noindent As a matter of terminology, we call $u$ a {\em singular control} and $\gamma_u$ the associated {\em singular} (or, equivalently, {\em abnormal}) {\em curve}.
	\end{defi}
	
	As explained for instance in~\cite{AgrSomeOpen} and~\cite[Section 10.2]{Montgomery}, the {\em Sard} (or {\em Morse-Sard}) {\em problem} concerns the following question: is it true that the singular set $\abn{}$ is negligible  in $\G$? More generally, how large can it be?  Remember that the Morse-Sard theorem for a smooth map defined on a finite dimensional manifold states that the set of critical values of the map has zero measure. However, this is no longer true in case the domain manifold is infinite-dimensional. The relevance of the Sard problem in sub-Riemannian geometry stems from the well-known influence  that  singular curves have on the regularity of geodesics, the regularity of the  distance and of its spheres, the heat diffusion, the analytic-hypoellipticity of sub-Laplacians, etc.
	
	Answers to the Sard problem are at the moment only partial. Building on techniques by L.~Rifford and E.~Tr\'elat~\cite{RiffTre}, A.~Agrachev  \cite{AgraAny} proved that, for general sub-Riemannian manifolds, singular curves that are also length-minimizing are contained in a closed nowhere dense set. A similar result has been obtained in \cite{Boa_Bar} by D.~Barilari and the first author in the more general case of control systems that are {\em affine} in the control, i.e., admitting a drift. In \cite{Vit_Sard}, the authors prove the negligibility of $\abn{}$ in Carnot groups of step 2 as well as in some other cases, some of which will be mentioned below. A detailed study of the singular set has been carried out in \cite{BelottoRifford,BelottoFPR} for 3-dimensional analytic sub-Riemannian manifolds with 2-dimensional analytic horizontal distributions: it turns out that such a set has Hausdorff dimension 1 and, actually, it is a semi-analytic curve.  Other partial or related results are contained in \cite{ZZ,RiffTre,LLMV_GAFA,LLMV_CAG,Gent_Hor,OttaVitto}. Different approaches to study singular curves are found e.g. in \cite{Chit_06, Chit_08,Bon_Kupka,Boa_Sig}, where the authors establish some regularity results that hold for the {\em generic} control system. Another line of investigation is pursued e.g. in \cite{ABL, Gent_Hor, BL_Hor}, where singular curves are analyzed through a topological viewpoint, building on variational methods \`a la Morse.
	
	The main results of the present paper are the following theorems.	
	
	\begin{thm}\label{thm:S_r2_s4}
		Let $\G$ be a Carnot group of rank $2$ and step $4$. Then, $\abn{}$ is a sub-analytic set of codimension at least $3$ in $\G$. 
	\end{thm}
	
	\begin{thm}\label{thm:S_r3_s3}
		Let $\G$ be a Carnot group of rank $3$ and step $3$. Then, $\abn{}$ is a sub-analytic set of codimension at least $1$ in $\G$. 
	\end{thm}
	
Theorem \ref{thm:S_r2_s4} was proved in \cite{LLMV_CAG}  for the {\em free}  Carnot group of rank 2 and step 4, see also \cite[Section 5.1]{Vit_Sard}. Recall that a Carnot group is free if the only relations imposed on its Lie algebra are those generated by the skew-symmetry and  Jacobi's identity. Also Theorem~\ref{thm:S_r3_s3} is  known for the free group of rank 3 and step 3, see  \cite[Section 5.1]{Vit_Sard}.  We however believe that the main novelty does not lie in the results {\em per se}, but rather in the techniques we exploit. The proofs given in \cite{LLMV_CAG,Vit_Sard} are purely algebraic and both rely on the so-called Tanaka prolongation of the Lie algebra of $\G$. In order for the strategy in \cite{LLMV_CAG,Vit_Sard} to work, it is necessary that the prolongation is long enough and, as a matter of fact, this does not happen in general. On the contrary, our dynamical-systems oriented approach can in principle be pursued in any Carnot group. Let us describe it.

Recall that each singular control $u$ is associated with a covector $\lambda\in\mfk g^*$  in such a way that $\lambda$ annihilates the image of $dF_e(u)$; since this image always contains $\mfk g_1$ (see~\eqref{eq:inclusion}), we actually have $\lambda\in\mfk g_1^\perp$. We use the necessary condition  given by Proposition~\ref{prop:criterion_abnormal} below to show that the primitive $w$ (see Definition~\ref{defi:primitive}) of the control $u$ is a concatenation (Definition~\ref{defi:concatenation}) of trajectories of a suitable dynamical system in $\R^r\equiv \mfk g_1$; $w$ can switch between different trajectories only at the equilibrium points of the dynamical system. When the group $\G$ is either  as in Theorem \ref{thm:S_r2_s4} or as in Theorem~\ref{thm:S_r3_s3}, the dynamical system is linear and, since the primitive has to start at the origin, one can classify all the singular curves associated with $\lambda$. The dynamical systems, of course, depend on $\lambda\in \mfk g_1^\perp$: an important part of our work consists in stratifying  $\mfk g_1^\perp$ as the finite union of sub-varieties $\Lambda_i$ in such a way that the dynamical systems associated with elements of each (fixed) $\Lambda_i$ are all conjugate. Eventually, the set $\abn{\Lambda_i}$ made by the union of all singular curves associated with elements of $\Lambda_i$ is sub-analytic,  and its codimension can be explicitly bounded. In particular, this codimension is at least 1 provided the codimension (in $\mfk g^*$) of $\Lambda_i$ is strictly greater than the dimension, in $\G$, of the set that can be reached by (lifts to $\G$ of) concatenations of trajectories of the dynamical system, associated with any $\lambda\in\Lambda_i$, that start at the origin.

We  believe that, in Theorem~\ref{thm:S_r3_s3}, the bound 1 on the codimension of $\abn{}$ can be improved and we conjecture that it holds  with a lower bound 3 (see \cite{OttaVitto} for an analogous open question in step 2 Carnot groups).  We are able to prove our conjecture at least when $\G$ is the free Carnot group of rank 3 and step 3. 

	\begin{thm}\label{thm:S_r3_s3free}
		Let $\G$ be the free Carnot group of rank $3$ and step $3$. Then, $\abn{}$ is a sub-analytic set of codimension $3$ in $\G$. 
	\end{thm}

The computation of a better bound on the codimension of $\abn{}$ reduces to the computation of the codimension of each $\abn{\Lambda_i}$ and is in principle possible with our techniques. It  requires some extra algebraic work  and, since we were not interested in obtaining better bounds on the codimension of $\abn{}$, we completed this task for the free group only. 

Another interesting feature of our approach is that it allows for a classification of singular curves revealing also their very shapes and their possible singularities. In particular, we recover many of the most exotic known examples of singular curves, see Remarks~\ref{rem:GoleKaridi} and~\ref{rem:ex1}, as well as new ones as in Remarks~\ref{rem:ex2} and~\ref{rem:ex3}.

When the rank $r$ and step $s$ of $\G$ satisfies
\begin{itemize}
\item either $r=2$ and $s\geq 5$
\item or $r=3$ and $s\geq 4$
\item or $r\geq 4$ and $s\geq 3$
\end{itemize}
the Sard problem is open. One can nevertheless  set up our approach and see that singular curves are again concatenations of trajectories of suitable dynamical systems; however, such systems are polynomial with degree two or more, and their study gets much harder. In Section \ref{sec:r2s5} we briefly discuss the situation in the case of  Carnot groups of rank 2 and step 5, where the involved dynamical systems are quadratic. Notice that a dynamical-systems  approach appears, although for different purposes, also in \cite{BCJPS}.

The paper is structured as follows. In Section~\ref{sec:preliminari} we discuss the preliminary material and we show how to derive the dynamical systems involved in our analysis; as an introductory warming up, we also use our dynamical approach to study the Sard problem in Carnot groups of rank 2 and step 3, see Section~\ref{sec:CG_r2_s3}. Theorems~\ref{thm:S_r2_s4} and~\ref{thm:S_r3_s3} are proved, respectively, in Sections~\ref{sec:CGr2s4} and~\ref{sec:r3s3}, while Theorem~\ref{thm:S_r3_s3free} is demonstrated in Section~\ref{sec:proof_free}. Finally, Section~\ref{sec:r2s5}  contains some musings  about    Carnot groups of rank 2 and step 5.

	\section{Preliminaries}\label{sec:preliminari}
 
Let $\G$ be a Carnot group as introduced in Section~\ref{sec:intro}. We consider on $\G$ the exponential map $\exp:\mfk{g}\to \G$, which is a real analytic diffeomorphism by, e.g, \cite[Theorem 1.2.1]{Corw_Greenleaf}. We also denote by $\cdot$ the group law in $\G$ and we define, given $g\in \G$, the left-translation map $L_g:\G\to \G$ by $L_g(h)=g\cdot h$.
	
	Let $n:=\dim(\G)$ and let $X_1,\dots, X_n$ be a basis of $\mfk{g}$ such that $X_1,\dots, X_r$ is a basis of $\mfk{g}_1$. When necessary, we tacitly identify $\mfk g_1$ and $ (dL_{g})_e\mfk g_1$, $g\in\G$, so that the elements $X_j$'s can be thought of as left-invariant vector fields on $\G$.  	
	We define a sub-Riemannian structure on $\G$ considering on $\mfk{g}_1$ the Riemannian metric that makes $X_1,\dots,X_r$ an orthonormal system.
	
	\begin{defi}\label{defi:defi_adm_curves}
	Let $\gamma:[0,1]\to\G$ be absolutely continuous and such that $\gamma(0)=e$. We say that $\gamma$ is an admissible curve if $\dot{\gamma}(t)\in \mfk{g}_1$ for a.e. $t\in [0,1]$ and $\mathrm{length}(\gamma):=\int_0^1|\dot{\gamma}(t)|dt<+\infty$, where we denoted by $|\cdot|$ the norm on $\mfk{g}_1$ induced by the fixed Riemannian metric.
	\end{defi}
	
	Let $u\in L^1([0,1],\mfk{g}_1)$ and let $\gamma_u:[0,1]\to \G$ be the curve solving a.e. on $[0,1]$ the ODE:
	\be\label{ed:ODE_admcurves}
		\dot{\gamma}(t)=(dL_{\gamma(t)})_eu(t),\ \ \gamma(0)=e.
	\ee 
	Then $\gamma_u$ is admissible. Conversely if $\gamma:[0,1]\to \G$ is an absolutely continuous curve satisfying \eqref{ed:ODE_admcurves} for some element $u\in L^1([0,1],\mfk{g}_1)$, then $\gamma$ is admissible and $u$ is its associated control. In coordinates, i.e. identifying $\mfk{g}_1$ with $\R^r=\mathrm{span}_\R\{X_1,\dots,X_r\}$, admissible curves are parametrized a.e. on $[0,1]$ by the integral curves of the ODE: 
	\be\label{eq:contr_sys}
		\dot{\gamma}(t)=u_1(t)X_1(\gamma(t))+\dots+u_r(t)X_r(\gamma(t)),\ \ \gamma(0)=e,
	\ee
	where $u\in L^1([0,1],\R^r)$.
	
	The notion of primitive of a control will play a basic role in the rest of the paper; we state it here.
	
		  \begin{defi}\label{defi:primitive}
	  Let $u\in L^1([0,1],\R^r)$. We call primitive of $u$ the function $w\in AC([0,1],\R^r)$ defined by:
	 	 \be
	  		w(t):=\int_0^t u(\tau)d\tau
	  	\ee
		for every $t\in [0,1]$. If we denote by $\pi_{\mfk{g}_1}$ the projection of $\mfk{g}$ onto $\mfk{g}_1$, we see that $w(t)=\pi_{\mfk{g}_1}(\exp^{-1}(\gamma(t)))$ for a.e. $t\in [0,1]$. In particular, once the function $w$ is known, $\gamma_{\dot{w}}$ is determined integrating \eqref{eq:contr_sys} with $u=\dot w$.
	  \end{defi}
	  
	\subsection{Elements of chronological calculus}	  
	   
 Singular curves are introduced in terms of the differential of the endpoint map in \eqref{eq:endpoint_intro}: in this section we introduce the formalism of the chronological calculus needed for its study. Chronological calculus is in essence an operatorial calculus introduced in \cite{Ag_Chron}, whose main properties we now recall. We identify points $g\in \G$ with homomorphisms of $C^\infty(\G)$ onto $\R$ by the formula $gf:=f(g)$, while we identify diffeomorphisms $P$ of $\G$ with automorphisms of $C^\infty(\G)$, i.e. with maps $f\mapsto Pf:=f(P(\cdot))\in C^\infty(\G)$. Tangent vectors at $g\in\G$ are identified with linear functionals on $C^\infty(\G)$ that satisfy the Leibniz rule: if $v\in T_g\G$ and $g(t)$ is a curve on $\G$ such that $g(0)=g$ and $\dot{g}(0)=v$, then
	\be
		\begin{aligned}
			v:C^\infty(\G)&\to \R,\\
				vf&:=\frac{d}{dt} f(g(t))\bigg|_{t=0}.
		\end{aligned}
	\ee
	Finally, we treat a smooth vector field $V$ as the derivation of the algebra $C^\infty(\G)$ given by $f\mapsto Vf$ for every $f\in C^\infty(\G)$.  We denote by $\Vc(\G)$ the set of all smooth vector fields on $\G$. Given  $t_I,t_F\in \R$, a non-autonomous vector field on $\G$, or simply a vector field on $\G$, is a measurable and locally bounded family $t\mapsto V_t$ for $t\in [t_I,t_F]$ and $V_t\in \Vc(\G)$ for every $t\in [t_I,t_F]$. We also agree that, in chronological notations, compositions are indicated by $\circ$ and are read from left to right. For more details, we refer the interested reader to \cite[Chapter 2]{Ag_Book} and to \cite{KipkaL}.
	
	Let $t_0\in [t_I,t_F]$. The flow of a vector field $V_t$ is a family of diffeomorphisms $(P^t_{t_0})$ on $\G$, $t\in [t_I,t_F]$, defined by the Cauchy problem:
	\be\label{eq:C_P}
		\left\{\begin{aligned}
			&\frac{d}{dt}P^t_{t_0}(g_0)=V_t(P^t_{t_0}(g_0)),\\
			&P^{t_0}_{t_0}(g_0)=g_0
		\end{aligned}\right.
	\ee
	for every $g_0\in \G$.  The assumptions on the family $(V_t)_{t\in [t_I,t_F]}$ imply that the solution to \eqref{eq:C_P} exists and is unique, at least locally.
	
	\begin{defi}\label{defi:right_chron_exp}
		Given $t_0\in [t_I,t_F]$ and a vector field $(V_t)_{t\in [t_I,t_F]}$, we define the (time-$t$ right) chronological exponential $\eexp\int_{t_0}^t V_\tau d\tau$ of $V$ as the diffeomorphism of $\G$  given by the formula 
		\be\label{NUMO}
			\eexp\int_{t_0}^t V_\tau d\tau:=P_{t_0}^t,
		\ee 
		where $P^t_{t_0}$ is defined as in \eqref{eq:C_P}. 
	\end{defi}
	
	Notice that $P^t_{t_0}$ solves the Cauchy problem $\frac{d}{dt} P^t_{t_0}=P^t_{t_0}\circ V_t$ on the space of operators on $C^\infty(\G)$, and that, if we want to include the initial datum $g_0\in \G$, in the formalism of chronological calculus we write $\frac{d}{dt}\left(g_0\circ  P^t_{t_0}\right)=g_0\circ P^t_{t_0}\circ V_t$.  Integrating iteratively the differential equation in \eqref{NUMO}, we may formally expand $P^t_{t_0}$ 
in the following Volterra series:
	\be
\label{eq:rightcronexp}
	\begin{aligned}
	P^t_{t_0} &=\mathrm{Id}+\sum_{k=1}^\infty \int_{\Sigma_k(t_0,t)} V_{\tau_k}\circ \dots \circ V_{\tau_1}d\tau_k\dots d\tau_1,\ \ &t\ge t_0,\\
	P^t_{t_0} &=\mathrm{Id}+\sum_{k=1}^\infty(-1)^k \int_{\Xi_k(t,t_0)} V_{\tau_k}\circ \dots \circ V_{\tau_1}d\tau_k\dots d\tau_1,\ \ &t< t_0.
	\end{aligned} 
\ee 
where 
\be
	\begin{aligned}
		\Sigma_k(t_0,t):=\{(\tau_1,\dots,\tau_k)\in\R^k\mid t_0\le\tau_k\le\dots\le\tau_1\le t\}&\ \ &\text{if }t\geq t_0,\\
		\Xi_k(t,t_0):=\{(\tau_1,\dots,\tau_k)\in\R^k\mid t\le\tau_1\le\dots\le\tau_k\le t_0\}&\ \ &\text{if }t< t_0.
	\end{aligned}
\ee
We also agree that $\Sigma_k(t):=\Sigma_k(0,t) $, $\Xi_k(t):=\Xi_k(t,0)$ and  $\Sigma_k:=\Sigma_k(1)$,
that is the $k$-th dimensional simplex. 

	\begin{remark}
		The equations in \eqref{eq:rightcronexp} are to be read as formal Volterra series. Indeed, as a consequence of Borel's Lemma \cite[Theorem 1.2.6]{Hor_Diff_Op}, these series are never convergent on $C^\infty(\G)$ in the weak sense unless $V_t\equiv 0$. This causes no harm to the rigour of our arguments, since we will only deal with finitely many terms in these expansions.
	\end{remark}

	\begin{remark}\label{rem:infinite}
	We will need to deal in the paper with vector fields $t\mapsto V_t$ well-defined for all times $t\in \R\cup \{\pm \infty\}$ and, accordingly, with chronological exponentials where either $t$ or $t_0$ is equal to $\pm \infty$. In these cases, denoting by $P_{t_0}^t$ the flow of $V_t$ as in~\eqref{NUMO}, one should read
 \be
 \eexp\int_{t_0}^{\pm\infty}V_\tau d\tau:=\lim_{t\to\pm\infty}P_{t_0}^t, \ \ \ \  \eexp\int_{\pm\infty}^{t_0}V_\tau d\tau:=\lim_{t\to\pm\infty} (P_{t_0}^{t})^{-1}.
 \ee
\end{remark}
	
	Let $B$ be a diffeomorphism of $\G$. The tangent map $B_*$ acts on vectors $v\in T_g\G$ as a composition $B_*v=v\circ B\in T_{B(g)}\G$. Then, if $V\in\Vc(\G)$, the action of $B_*$ on $V$ is given by
	\be\label{eq:diff_chron_calc}
		B_*V=B^{-1}\circ V\circ B,
	\ee
	that is, $B_*V$ is the standard push-forward map. The vector field $(\mathrm{Ad}B)V$ is defined by the formula 
	\be
		(\mathrm{Ad}B)V=B\circ V\circ B^{-1}
	\ee 
	and we have the identity $\mathrm{Ad}(B^{-1})=B_*$.
	
	Given a flow $P^t_{t_0}:=\eexp\int^t_{t_0} V_\tau d\tau$, we want to write down an ODE describing the evolution of $\mathrm{Ad}P^t_{t_0}$. This differential equation is meant at the level of operators on the Lie algebra  of the smooth vector fields on $\G$. 
For every $X\in\Vc(\G)$ we have:
	\be\label{eq:ad_operator}
			\frac{d}{dt}\mathrm{Ad}P^t_{t_0}X=P^t_{t_0}\circ \left(V_t\circ X-X\circ V_t\right)\circ P^{-t}_{t_0}=(\mathrm{Ad}P^t_{t_0})[V_t,X]=(\mathrm{Ad}P^t_{t_0})\mathrm{ad}V_tX,
	\ee
	where $\mathrm{ad}$ denotes  the standard left Lie multiplication. By the arguments in \cite[\S 2.5]{Ag_Book} we see that $\mathrm{Ad}P^t_{t_0}$ is the unique solution to the Cauchy problem
	\be
		\frac{d}{dt} A^t_{t_0}=A^t_{t_0}\circ \mathrm{ad} V_t, \ \ A_{t_0}^{t_0}=\mathrm{Id},
	\ee
	and this allows for the definition:
	\be\label{eq:id_Ad_ad}
		\eexp\int_{t_0}^t\mathrm{ad}V_\tau d\tau:=\mathrm{Ad}\left(\eexp\int_{t_0}^t V_\tau d\tau\right).
	\ee

	\subsection{The differential of the endpoint map} 
	Given $v\in \R^r$, we introduce the short-hand notation $X_{v}:=\sum_{i=1}^rv_i X_i\in \mfk{g}_1$. 
	\begin{defi}
		For every $t\in [0,1]$, we define the map
		\be\label{eq:endp_t}
			\begin{aligned}
				F_e^t:L^1([0,1],\R^r)&\to \G\\
				         u &\mapsto \gamma_u(t).
			\end{aligned}
		\ee
	\end{defi}
	The endpoint map $F_e$ in \eqref{eq:endpoint_intro} coincides with $F_e^1$, and for every $t\in [0,1]$ the map $F^t_e$ is given by the formula:
	\be\label{eq:endp}
		F_e^t(u)=e\circ\eexp \int_0^t X_{u(\tau)}d\tau.
	\ee 
	Let $v\in L^1([0,1],\R^r)$. We compute $F_{e}(u+v)$ as a perturbation of $F_{e}(u)$. By \eqref{eq:id_Ad_ad} we define, for $t\in [0,1]$,
	\be\label{eq:gut}
		\begin{aligned}
			g^{u,t}_{v(t)}:&=\mathrm{Ad}\left(\eexp\int_0^t X_{u(\tau)}d\tau \right)X_{v(t)}
					=\left(\eexp\int_0^t\mathrm{ad}X_{u(\tau)}d\tau\right) X_{v(t)},
		\end{aligned}
	\ee
	and by the variations' formula in \cite[Section 2.7]{Ag_Book} we write:
	\be\label{eq:var_for_exp}\begin{aligned}
		F_{e}(u+v)&=e\circ\eexp\int_0^1X_{u(t)}+X_{v(t)}dt
		\\& = e\circ \eexp\int_0^1 \mathrm{Ad}\left(\eexp\int_0^t 
		X_{u(\tau)}d\tau \right)X_{v(t)}dt \circ \eexp \int_0^1 X_{u(t)}dt
		\\&=e\circ \eexp\int_0^1 g^{u,t}_{v(t)}dt \circ \eexp \int_0^1 X_{u(t)}dt.
	\end{aligned}\ee
	The  derivative $d_uF_e(v)$ is given by the first-order term in the series expansion with respect to $v$ of \eqref{eq:var_for_exp}, that is
	\be\label{eq:dif_var_for} 
		d_uF_e(v)=e\circ \int_0^1g^{u,t}_{v(t)}dt\circ \eexp \int_0^1 X_{u(t)}dt.
	\ee	
	Notice that, in the classical formalism of differential geometry this means that
	\be
		d_uF_e(v)=\left(\eexp \int_0^1 X_{u(t)}dt\right)_*\left(\int_0^1g^{u,t}_{v(t)}dt(e)\right),
	\ee
	so that $d_uF_e(v)$ is nothing but the push-forward, via the tangent map $(\eexp \int_0^1 X_{u(t)}dt)_*$, of the tangent vector $\int_0^1g^{u,t}_{v(t)}dt(e)\in \mfk{g}$. 
	
	The image of the differential $d_uF_e$ is then described, up to a diffeomorphism, by the mapping
	\be\label{eq:differential}\begin{aligned}
		G_e^u:L^1([0,1],\R^r)&\to \mfk{g}, \\
				v&\mapsto \int_0^1g^{u,t}_{v(t)}dt(e),
	\end{aligned}\ee
	and it follows by construction that the differential $d_uF_e$ is surjective if and only if $\IM(G_e^u)=\mfk{g}$. Owing to \eqref{eq:gut} and \cite[equation (2.23)]{Ag_Book}, $G_e^u(v)$ admits the expansion:	
	\be\label{eq:G_e^u}
		G_e^u(v)=\sum_{j=1}^s\int_{\Sigma_j}\left(\mathrm{ad}X_{u(\tau_j)}\circ\dots\circ  \mathrm{ad}X_{u(\tau_2)}  \right)X_{v(\tau_1)}d\tau_j\dots d\tau_1(e),
	\ee
	where  the sum runs over a finite number of indices because $\mfk{g}$ is nilpotent of step $s$, and the first term in \eqref{eq:G_e^u} is to be intended as $\int_0^1X_{v(\tau_1)}d\tau_1(e)$. A useful characterization of the image of $G_e^u$ is provided in the next proposition (compare with \cite[Proposition 2.3]{Vit_Sard}).
	
	\begin{prop}\label{prop:alt_char_image_diff}
	The following formula holds:
			\be\label{eq:alt_char_im}
			\IM(G_e^u):=\Span_{Y\in \mfk{g}_1,t\in [0,1]}\left\{ \sum_{j=0}^{s-1}\int_{\Sigma_j(t)}\left(\mathrm{ad}X_{u(\tau_{j})}\circ\dots\circ  \mathrm{ad}X_{u(\tau_1)}  \right)Yd\tau_{j}\dots d\tau_1(e) \right\},
		\ee
		where, for every $Y \in \mfk{g}_1$, the 0-th  term in the summation simply denotes $Y(e)$.
	\end{prop}
	
	\begin{proof}
		By \eqref{eq:G_e^u}, we have:
		\be\
			\IM(G_e^u)=\left\{
			\sum_{j=1}^s\int_{\Sigma_j}\left(\mathrm{ad}X_{u(\tau_j)}\circ\dots\circ  \mathrm{ad}X_{u(\tau_2)}  \right)X_{v(\tau_1)}d\tau_j\dots d\tau_1(e)
			\mid v\in L^1([0,1],\R^r)\right\}.
		\ee
		To establish the $\subset$ inclusion in \eqref{eq:alt_char_im}, we  
		notice that any element in $\IM(G_e^u)$ can be seen as the limit of finite sums of elements in the right-hand side of \eqref{eq:alt_char_im}, which in turn is a closed set that contains all of its limit points. 
		
		To deduce the $\supset$ inclusion in \eqref{eq:alt_char_im}, we fix instead a basis $(e_i)_{i=1}^r$ of $\R^r$, so that $X_{e_i}=X_i$ for $1\le i\le r$.	We fix $t\in [0,1)$ (the case $t=1$ can be treated similarly) and, for $n$ large enough, we consider $\psi_n:=n\chi_{[t,t+\frac1n]}$ to see that
		\be
			\sum_{j=0}^{s-1}\int_{\Sigma_j(t)}\left(\mathrm{ad}X_{u(\tau_j)}\circ\dots\circ  \mathrm{ad}X_{u(\tau_1)}  \right)X_id\tau_j\dots d\tau_1(e)=\lim_{n\to \infty}G_e^u(\psi_n e_i)\in \IM(G_e^u)
		\ee
		since $\IM(G_e^u)$ is closed as well, and we conclude.
	\end{proof}

		One can consider the elements of the right-hand side of~\eqref{eq:alt_char_im} corresponding to $t=0$ to see that
		\be\label{eq:inclusion}
		\mfk{g}_1\subset \IM(G_e^u).
		\ee
		Moreover, one can write $\IM(G_e^u)=\mfk g_1\oplus\mfk R_u$, where		
		\be\label{eq:subsp_A}
		\mfk{R}_u:=\Span_{Y\in \mfk{g}_1,t\in [0,1]}\left\{\sum_{j=1}^{s-1}\int_{\Sigma_j(t)}\left(\mathrm{ad}X_{u(\tau_j)}\circ\dots\circ  \mathrm{ad}X_{u(\tau_1)}  \right)Yd\tau_j\dots d\tau_1(e)\right\}.
	\ee

	We defined a singular control $u$ as a critical point of $dF_e$, i.e. as an element $u\in L^1([0,1],\R^r)$ such that the map $d_uF_e:L^1([0,1],\R)\to \mfk{g}$ is not surjective, see Definition~\ref{defi:abnormal_set}. With our discussion we have shown the following alternative characterization.
		
	\begin{prop}\label{prop:singular}
	A control $u\in L^1([0,1],\R^r)$ is singular if and only if the subspace $\mfk{R}_u$ is a proper subspace of $\mfk{g}_2\oplus \dots\oplus \mfk{g}_{s}$.
	\end{prop}

	\begin{remark}\label{rem:step_two}
		For a Carnot group $\G$ of step $2$, a control $u\in L^1([0,1],\R^r)$ is singular if and only if the family $\left\{X_{u(t)}\mid t\in [0,1]\right\}$ spans at most an $(r-2)$-dimensional subspace. Indeed, if this is not the case, we see that $\mfk{R}_u=\mfk{g}_2$.
		This is one of the key observations leading to the proof of the  
		Sard property for Carnot groups of step $2$ (see \cite{Gent_Hor,Vit_Sard}).
	\end{remark}
	
	\subsection{A dual point of view}
	
	\begin{defi}\label{defi:brackets}
		Given $k\in \N$ and $i_1,\dots, i_k\in \{1,\dots,r\}$, we define
		\be\label{eq:brackets}
			X_{i_1\dots i_k}(e):=[X_{i_1},[\dots,[X_{i_{k-1}},X_{i_k}]\dots]](e).
		\ee 
		By multi-linearity of the Lie brackets, recalling that for $v\in \R^r$ we defined $X_v\in \mfk{g}_1$ as the sum $\sum_{i=1}^rv_iX_i$, \eqref{eq:brackets} can be extended to expressions of the form $X_{v_1\dots v_k}(e)$ for arbitrary vectors $v_1,\dots, v_k\in \R^r$. We also use round brackets to indicate the priority of nested commutators. In this way, any commutator is identified with a word $J=(j_1,\dots,j_k)$ with letters in the alphabet $\{1,\dots,r,(,),v\mid v\in \R^r\}$. For example, assuming $r=2$, we have
		\be
			X_{(12)(112)}(e)=[[X_1,X_2],[X_1,[X_1,X_2]]](e).
		\ee
		Given a covector $\lambda\in \mfk{g}^*$ and a string $J$ as above, we define $\lambda_J:=\left\langle  \lambda,X_J(e)\right\rangle$. If $J_1,\dots,J_{\mathrm{dim}(\mfk{g}_k)}$ are strings such that $X_{J_1},\dots, X_{J_{\mathrm{dim}(\mfk{g}_k)}}$ form a basis of $\mfk{g}_k$ for $k\in \{1,\dots, s\}$, $\lambda_{J_1},\dots,\lambda_{J_{\mathrm{dim}(\mfk{g}_k)}}$ are the coordinates of $\lambda$ on $\mfk{g}_k^*$. 
\end{defi}
		
	It follows from Proposition~\ref{prop:singular} that a control $u\in L^1([0,1],\R^r)$ is singular if and only if there exists a nonzero $\lambda\in \mfk{g}_2^*\oplus\dots\oplus \mfk{g}_s^*$  such that $\lambda \in \IM(G_e^u)^\perp$: in fact, the inclusion $\mfk{g}_1\subset \IM(G_e^u)$ yields that any $\lambda \in \IM(G_e^u)^\perp$ has zero projection on $\mfk{g}_1^*$. 
	
	\begin{defi}\label{defi:abnGLambda}
		For a given subset $\Lambda\subset \mfk{g}^*$, we define 
		\be
			\abn{\Lambda}:=\left\{  \gamma_u(1)\mid u\in L^1([0,1],\R^r),\, \text{and there exists $\lambda\in \Lambda$ such that $\lambda\in \IM(G_e^u)^\perp$} \right\}\subset \G,
		\ee
		that is $\abn{\Lambda}$ contains all the final points of singular curves $\gamma_u$ issuing from the origin $e\in \G$, and associated with some covector $\lambda\in \Lambda$ orthogonal to $\IM(G_e^u)$.
	\end{defi}

	\begin{remark}\label{rem:normone}  
	The condition $\lambda\in\IM(G_e^u)^\perp$ is projectively invariant. Given any quadratic norm $\|\cdot\|$ on $ 
	\mfk{g}^*$, we can always assume that
	\be
		\lambda\in\mathbb{S}(\mfk{g}_2^*\oplus\dots\oplus \mfk{g}_s^*):=\left\{ \xi\in  \mfk{g}_2^*\oplus\dots\oplus \mfk{g}_s^*\mid \|\xi\|=1 \right\}.
	\ee 
	\end{remark} 
	It follows from \eqref{eq:subsp_A} that $\lambda\in \mfk{g}_2^*\oplus\dots\oplus \mfk{g}_s^*$ belongs to $\IM(G_e^u)^\perp$ if and only if 
	\be
		\sum_{k=1}^{s-1}\int_{\Sigma_{k}(t)}\lambda_{u(\tau_{k})\dots u(\tau_1)j} d\tau_{k}\dots d\tau_1=0,
	\ee
	for every $j=1,\dots,r$ and all $t\in [0,1]$. By differentiating with respect to $t$ we obtain	
	\be\label{eq:first_cond_perp}
		\sum_{k=1}^{s-1}\int_{\Sigma_{k-1}(t)}\lambda_{u(\tau_{k-1})\dots u(\tau_1)u(t)j} d\tau_{k-1}\dots d\tau_1=0,
	\ee
	for every $j=1,\dots,r$ and a.e. $t\in [0,1]$. Owing again to the multi-linearity of the Lie brackets, \eqref{eq:first_cond_perp} implies that 
	\be\label{eq:sec_cond_perp}
		\sum_{i=1}^r u_i(t)\left(\sum_{k=1}^{s-1}\int_{\Sigma_{k-1}(t)}\lambda_{u(\tau_{k-1})\dots u(\tau_1)ij} d\tau_{k-1}\dots d\tau_1\right)=0, \ \ j=1,\dots,r
	\ee
	for a.e. $t\in [0,1]$. This discussion proves the following result.
	
	\begin{prop}\label{prop:criterion_abnormal}
		Given $u\in L^1([0,1],\R^r)$, we define the skew-symmetric matrix $\mathscr{M}_u(\lambda,t)\in M_r(\R)$ by:
	\be\label{eq:matrix_abnormal}
		\mathscr{M}_u(\lambda,t)_{ij}:=\sum_{k=1}^{s-1}\int_{\Sigma_{k-1}(t)}\lambda_{u(\tau_{k-1})\dots u(\tau_1)ij} d\tau_{k-1}\dots d\tau_1,\ \ i,j=1,\dots,r.
		\ee
		Then a control $u\in L^1([0,1],\R^r)$ is singular if and only if there exists  $\lambda\in \mathbb S(\mfk{g}_2^*\oplus\dots\oplus \mfk{g}_s^*)$ such that
		\be
			u(t)\in \ker(\mathscr{M}_u(\lambda,t))
		\ee
		for a.e. $t\in [0,1]$. 
	\end{prop} 
	
	\begin{remark}[Goh condition on Carnot groups of rank $2$]\label{rem:Goh} 
	Given a singular trajectory $\gamma_u$ contained in a Carnot group $\G$ of rank $2$, it is not difficult to see that $\mfk{g}_2\subset \IM(G_e^u)$ (see, e.g. \cite[Remark 2.8]{Vit_Sard}). In particular, every covector $\lambda\in\IM(G_e^u)^\perp$ is orthogonal to $\mfk{g}_2$ (equivalently $\lambda\in \mfk{g}_3^*\oplus\dots\oplus \mfk{g}_s^*$), i.e. $\lambda$ automatically satisfies the so-called {\em Goh condition}.
	\end{remark}

	Proposition~\ref{prop:criterion_abnormal} is of fundamental importance in our paper: indeed, it will allow us to study singular curves in terms of {\em concatenations} of trajectories of suitable dynamical systems. Let us fix some terminology: first, given a smooth vector field $V$ on $\R^r$, we call set of equilibria of the first-order differential system $\dot{x}=V(x)$, $x\in \R^r$, the set $\{x\in\R^r\mid V(x)=0\}$.

	\begin{defi}[Concatenation]\label{defi:concatenation}
		For a smooth vector field $V$ on $\R^r$ consider a differential system of the form
		\be\label{eq:model_diff_syst}
			\dot{x}=V(x),\ \ x\in \R^r.
		\ee
		We say that $w\in AC([0,1],\R^r)$ is a concatenation of the integral curves of \eqref{eq:model_diff_syst} if there exists an open set $I\subset [0,1]$ with the following properties:
		\begin{itemize}
			\item [(i)] write $I=\bigcup_{i} I_i$ as a finite or countable  disjoint union of open intervals. Then, for every $i$, $w(I_i)$ is contained in an integral curve of \eqref{eq:model_diff_syst};
			\item [(ii)] $w([0,1]\setminus I)$ is contained in the set of equilibria of \eqref{eq:model_diff_syst}.
		\end{itemize}
	\end{defi}
		
The differential systems involved in our analysis will depend on some parameter $\lambda$, typically in a sub-analytic fashion. We recall here the relevant definitions,  borrowed from \cite{Bier_Milm}.		
		
			\begin{defi}[Sub-analytic sets and functions] \quad

	\begin{itemize}
		\item [(a)] A set $X\subset M$ of a real analytic manifold $M$ is {\em semi-analytic} if, for every $x\in M$, there exists an open neighborhood $U$ of $x$ such that $X\cap U$ is a finite Boolean combination of sets $\{y\in U\mid f(y)=0\}$ and $\{y\in U\mid g(y)>0\}$, where $f,g:U\to \R$ are analytic functions.
		\item [(b)] Let $M$ be a real analytic manifold. A set $X\subset M$ is {\em sub-analytic} if, for every $x\in M$, there exist an open neighborhood $U$ of $x$, a real analytic manifold $N$ and a relatively compact semi-analytic subset $A\subset M\times N$ such that $X\cap U=\pi(A)$, where $\pi:M\times N\to M$ is the canonical projection.
		\item[(c)] Let  $M,N$ be real analytic manifolds. A function $f:M\to N$ is {\em sub-analytic} if its graph is a sub-analytic set in $M\times N$.
	\end{itemize}
	\end{defi}
	
	The image of a relatively compact sub-analytic set by a sub-analytic mapping
is sub-analytic.

	\subsection{Carnot groups of rank 2 and step 3}\label{sec:CG_r2_s3}
	
	As a warming up, we discuss Lie groups $\G$ of rank $2$ and step $3$. This case is already well-known in the literature, as $\G$ is either the 5-dimensional free group (where $\abn{}=\exp(\mfk g_1)$) or the 4-dimensional Engel group (where $\abn{}=\exp(\R X)$ for some $X\in\mfk g_1$).
	
	Pick a singular trajectory $\gamma_u$ and let $u\in L^1([0,1],\R^r)$ be the associated control. Since $u$ is singular, there exists $\lambda\in \mfk{g}_2^*\oplus \mfk{g}_3^*$ such that \eqref{eq:sec_cond_perp} holds. In fact $\mfk{g}_1\oplus \mfk{g}_2\subset\IM(G_e^u)$ by the Goh condition, Remark~\ref{rem:Goh}, and therefore $\lambda\in \mfk{g}_3^*$.

	  Following Definition~\ref{defi:brackets}, the skew-symmetric matrix $\mathscr{M}_u(\lambda,t)\in M_2(\R)$ in \eqref{eq:matrix_abnormal} is given by: 
	\be
		\mathscr{M}_u(\lambda,t)_{ij}=\int_0^t\lambda_{u(\tau_1)ij}d\tau_1=\lambda_{\int_0^tu(\tau_1)d\tau_1ij}=\lambda_{w(t)ij},\ \ i,j=1,2,
	\ee
	where $w$ is the primitive of $u$ (see Definition~\ref{defi:primitive}) and the second equality follows by the linearity of the map $v\mapsto \lambda_{vij}$ for every $v\in \R^2$. 
	By Proposition~\ref{prop:criterion_abnormal}, $u(t)\in \ker(\mathscr{M}_u(\lambda,t))$ a.e. $t\in [0,1]$, and then
	\be
		\mathrm{Pf}(\mathscr{M}_u(\lambda,t))=w_1(t)\lambda_{112}+w_2(t)\lambda_{212}=0
	\ee
	for every $t\in [0,1]$. By  differentiating this last relation we  finally deduce that 
	\be
		u_1(t)\lambda_{112}+u_2(t)\lambda_{212}=0
	\ee
	a.e. on $[0,1]$, meaning that $(u_1(t),u_2(t))$ is parallel to $(\lambda_{212},-\lambda_{112})$ for a.e. $t\in [0,1]$. 
	
	Let us now fix $\lambda\in\mfk{g}_3^*$.  
	Then all the primitives $w$ associated   with such a $\lambda$ (that is, such that $\lambda\in \IM(G_e^{\dot{w}})^\perp$ where $\dot w=u$) 
	are supported within the integral curves of the differential system
	\be
		\dot{x}(t)=v(\lambda),
		\ \ x\in \R^2,
	\ee
	where we denoted by $v$ the map (see Remark~\ref{rem:normone} for the definition of $\mathbb{S}(\mfk{g}_3^*)$):
	\be
		\begin{aligned}
			v:\mathbb{S}(\mfk{g}_3^*)&\to \R^2,\\
					\lambda&\mapsto \begin{pmatrix} \lambda_{212} \\ -\lambda_{112}\end{pmatrix}.
		\end{aligned}
	\ee
	Since $w(0)=0$ and for every $\lambda\in\mfk{g}_3^*$ the vector $v(\lambda)$ is not zero (for otherwise $\lambda$ itself would be zero), we conclude that the primitives $w$ associated with $\lambda$ are supported within the line $x(t)=v(\lambda)t$.
Accordingly, every singular curve $\gamma_{\dot{w}}\subset \G$ associated with $\lambda$ is supported within the one-dimensional analytic submanifold $\left\{\mathscr{L}(\lambda,t)\mid t\in \R\right\}\subset \G$, where 
\be
		\mathscr{L}(\lambda,t):=\exp\left(t(\lambda_{212} X_1-\lambda_{112}X_2)\right).
	\ee
	Noticing that also the assignment $\lambda\mapsto \mathscr{L}(\lambda,t)$ is analytic for every $t\in \R$, we conclude by standard transversality arguments (see e.g. \cite{Abra_Rob,Guil_Pol}) that $\abn{}$ is an analytic submanifold of codimension at least $3$ in $\G$ (we have taken into account that the projection of $\lambda$ onto $\mfk{g}_1^*\oplus \mfk{g}_2^*$ is $0$).
	
	\begin{prop}\label{prop:r2s3}
		Let $\G$ be a Carnot group of rank $2$ and step $3$. Then, $\abn{}$ is an analytic submanifold of codimension at least $3$ in $\G$.
	\end{prop}

	\section{Carnot groups of rank 2 and step 4}\label{sec:CGr2s4}
	 Let $u\in L^1([0,1],\R^2)$, $w\in AC([0,1],\R^2)$ be the primitive of $u$ and $\gamma_u$ be the singular trajectory associated with $u$. By Proposition~\ref{prop:singular}, the subspace $\mfk{R}_u$, generated by elements of the form 	\be\label{eq:subspace_non_surj_s2r4}
		[X_{w(t)},X_j](e)+\int_0^t[X_{w(\tau_1)},[X_{u(\tau_1)},X_j]]d\tau_1(e)+\iint\limits_{0\le\tau_2\le\tau_1\le t}[ X_{w(\tau_2)},[ X_{u(\tau_2)},[X_{u(\tau_1)},X_j]]](e)d\tau_2d\tau_1,
	\ee
	for a.e. $t\in [0,1]$ and $j=1,2$, is strictly contained in $\mfk{g}_2\oplus \mfk{g}_3\oplus \mfk{g}_4$ (compare with \eqref{eq:subsp_A}). 
	By the Goh condition (Remark~\ref{rem:Goh}) we deduce as in Section~\ref{sec:CG_r2_s3} 
	the existence of a  
	covector $\lambda\in \mathbb{S}(\mfk{g}_3^*\oplus \mfk{g}_4^*)$ such that, upon differentiating \eqref{eq:subspace_non_surj_s2r4} with respect to $t$, the identity
	\be
		\lambda_{w(t)u(t)j}+\int_0^t\lambda_{w(\tau_1)u(\tau_1)u(t)j}d\tau_1=0
	\ee
	holds for a.e. $t\in [0,1]$ and $j=1,2$. The skew-symmetric matrix $\mathscr{M}_u(\lambda,t)\in M_2(\R)$ in \eqref{eq:matrix_abnormal} is given by
	\be
		\mathscr{M}_u(\lambda,t)_{ij}=\lambda_{w(t)ij}+\int_0^t\lambda_{w(\tau_1)u(\tau_1)ij}d\tau_1,\ \ i,j=1,2,
	\ee
	and $u(t)\in\ker(\mathscr{M}_u(\lambda,t))$ a.e. $t\in [0,1]$ implies that:
	\be\label{eq:pf_r2_s4}
		\mathrm{Pf}(\mathscr{M}_u(\lambda,t))=\lambda_{w(t)12}+\int_0^t\lambda_{w(\tau_1)u(\tau_1)12}d\tau_1=0
	\ee
	for every $t\in [0,1]$. Notice that \eqref{eq:pf_r2_s4} is differentiable with respect to $t$, and gives
	\be\label{eq:pf_r2_s4_two}
		\sum_{i=1}^2u_i(t)\left(  \lambda_{i12}+\lambda_{w(t)i12}  \right)=0,\ \ \text{a.e. }t\in [0,1],
	\ee
	that is, we conclude that $(u_1(t),u_2(t))$ is parallel to $(\lambda_{212}+\lambda_{w(t)212},-\lambda_{112}-\lambda_{w(t)112})\in \R^2$ for a.e. $t\in [0,1]$. Let us recall that $\lambda_{1212}=\lambda_{2112}$ by Jacobi's identity.

	We fix $\lambda\in \mathbb{S}(\mfk{g}_3^*\oplus \mfk{g}_4^*)$. Forgetting  about possible parametrizations, we conclude from \eqref{eq:pf_r2_s4_two} that all the primitives $w$ 
	such that $\lambda\in \IM(G_e^{\dot{w}})^\perp$, are  concatenations (see Definition~\ref{defi:concatenation}) 	of  the integral curves of the differential system
	\be\label{eq:dif_sys_r2_s4}
		\dot{x}=M(\lambda) x+v(\lambda),\  \ x\in \R^2
	\ee
	where we introduced this time the mappings:
	\be
		\begin{aligned}
			M:\mathbb{S}(\mfk{g}_3^*\oplus \mfk{g}_4^*)&\to M_2(\R),\\
			\lambda&\mapsto \begin{pmatrix} \lambda_{2112} & \lambda_{2212}\\ 
								    -\lambda_{1112} & -\lambda_{2112}
						    \end{pmatrix},
		\end{aligned}\ \ \ \ \ \ 
		\begin{aligned}
			v:\mathbb{S}(\mfk{g}_3^*\oplus \mfk{g}_4^*)&\to \R^2,\\
			\lambda&\mapsto \begin{pmatrix} \lambda_{212} \\ -\lambda_{112}\end{pmatrix}.
		\end{aligned}
	\ee
	Notice that, since  the  primitives $w$ satisfy $w(0)=0$, one has to take into account only those concatenations starting at the origin. Observe also that $\trace(M(\lambda))=0$ for every $\lambda\in \mathbb{S}(\mfk{g}_3^*\oplus \mfk{g}_4^*)$, and that both the assignments $\lambda\mapsto M(\lambda)$ and $\lambda\mapsto v(\lambda)$ are analytic.
	
We stratify $\mathbb{S}(\mfk{g}_3^*\oplus \mfk{g}_4^*)$ according to $\rank(M(\lambda))$, and we consider the (pairwise disjoint) sub-analytic sets
	\be\label{eq:stratification_r2_s4}
		\begin{aligned}
			\Lambda_1&:=\left\{ \lambda\in \mathbb{S}(\mfk{g}_3^*\oplus \mfk{g}_4^*)\mid \det(M(\lambda))<0 \right\},\\
			\Lambda_2&:=\left\{ \lambda\in \mathbb{S}(\mfk{g}_3^*\oplus \mfk{g}_4^*)\mid \det(M(\lambda))>0 \right\},\\
			\Lambda_3&:=\left\{ \lambda\in \mathbb{S}(\mfk{g}_3^*\oplus \mfk{g}_4^*)\mid \rank(M(\lambda))=1 \right\},\\
			\Lambda_4&:=\left\{ \lambda\in \mathbb{S}(\mfk{g}_3^*\oplus \mfk{g}_4^*)\mid M(\lambda)=0 \right\}.
		\end{aligned}
	\ee
	We complete the proof of Theorem~\ref{thm:S_r2_s4}  analyzing separately each one of the above cases. Notice that the Jordan normal form $N$ of $M(\lambda)$ is constant on each 
	of the sets above, i.e. there exists $N=N(\Lambda_i)\in M_2(\R)$ such that, for every $\lambda\in \Lambda_i$, there exists $P(\lambda)\in GL_2(\R)$ such that
	\be\label{eq:first_norfor}
		N=P(\lambda)^{-1}M(\lambda)P(\lambda).
	\ee 
	Moreover, the mappings $\lambda\mapsto P(\lambda)$ and $\lambda\mapsto P(\lambda)^{-1}$ can be chosen to be sub-analytic on each one of the sets $\Lambda_i$.

	Up to a linear change of coordinates on $\R^2$ (not depending on time), of the form $z:=P(\lambda)^{-1}x$, it is therefore sufficient to study, for $\lambda\in \Lambda\in\{\Lambda_1,\Lambda_2,\Lambda_3,\Lambda_4\}$, the differential system
	\be\label{eq:diff_sys_r2_s4}
		\dot{z}=Nz+b(\lambda),\  \ z\in \R^2,
	\ee
	where we defined $b(\lambda):=P(\lambda)^{-1}v(\lambda)\in \R^2$, and the assignment $\lambda\mapsto b(\lambda)$ is  sub-analytic. 
	
	\begin{remark}\label{rem:wviachangeofcoord}
	In the sequel we will make an abuse of notation by identifying the primitive $w$ with  $P(\lambda)^{-1} w$.
	\end{remark}
	
	\begin{remark}\label{rem:changeofcoord}
		The change of coordinates $z=P(\lambda)^{-1}x$ induces a change in the basis $X_1,X_2$ of $\mfk{g}_1$. More specifically, assuming
		\be
			P(\lambda)^{-1}=\begin{pmatrix}
				p_{11}(\lambda) & p_{12}(\lambda)\\
				p_{21}(\lambda) & p_{22}(\lambda)
			\end{pmatrix},
		\ee
		we obtain
		\be
			\begin{aligned}
				X_1(\lambda)&:=p_{11}(\lambda)X_1+ p_{21}(\lambda)X_2,\\
				X_2(\lambda)&:=p_{12}(\lambda)X_1+ p_{22}(\lambda)X_2,
			\end{aligned}
		\ee 
		and the map $\lambda\mapsto (X_1(\lambda),X_2(\lambda))$ is sub-analytic.
	\end{remark}

	\subsection{Case I\texorpdfstring{: $\Lambda=\Lambda_1$}{}} In this case
	\be
		N=\begin{pmatrix} 1 & 0\\
					     0 & -1 \end{pmatrix}.
	\ee

	Given $\lambda\in \Lambda_1$, the integral curves of \eqref{eq:diff_sys_r2_s4} starting at $(z^0_1,z^0_2)$ are given by
	\be\label{eq:expl_dynamic_maxr}
		\left\{\begin{aligned}
			z_1(t) &= (e^t-1)b(\lambda)_1+e^tz^0_1,\\
			z_2(t) &= -(e^{-t}-1)b(\lambda)_2+e^{-t}z^0_2.
		\end{aligned}\right.
	\ee
	These trajectories are depicted in Figure~\ref{fig:rank2step4caso1}.
			\begin{figure}
		\includegraphics[scale=.6]{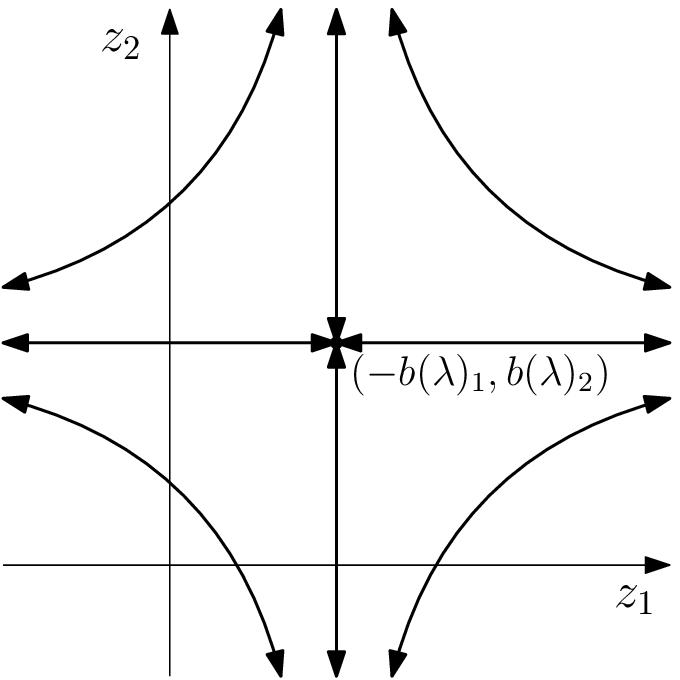}
		\caption{Trajectories of \eqref{eq:expl_dynamic_maxr}.}
		\label{fig:rank2step4caso1}
	\end{figure}
	Therefore, a trajectory of \eqref{eq:diff_sys_r2_s4} asymptotically approaches  the equilibrium $(-b(\lambda)_1,b(\lambda)_2)$ if and only if 
	\be\label{eq:cond_equilibrium}
\text{either  $b(\lambda)_1+z^0_1=0$ or $-b(\lambda)_2+z^0_2=0$.}
	\ee 
	
	If both the conditions in \eqref{eq:cond_equilibrium} are met, $z(t)$ remains indefinitely in the equilibrium. If only one of these conditions is satisfied, $z(t)$ approaches the equilibrium asymptotically, and only once (i.e., either in the limit as  $t\to+\infty$ or as $t\to -\infty$).
	
Since the concatenations we consider have to start at the origin, it is natural to introduce the sets 
	\be
	\begin{aligned}
		\Xi_1&:=\left\{ \lambda\in \Lambda_1 \mid b(\lambda)_1\ne 0,\, b(\lambda)_2\ne 0  \right\},\\
		\Xi_2&:=\left\{ \lambda\in \Lambda_1 \mid b(\lambda)_1= 0  \right\},\\
		\Xi_3&:=\left\{ \lambda\in \Lambda_1 \mid b(\lambda)_2= 0  \right\},
	\end{aligned}
	\ee
	whose union covers $\Lambda_1$
	
	For every $\lambda\in \Xi_1$, the solution to \eqref{eq:diff_sys_r2_s4} starting at the origin never crosses the equilibrium, not even asymptotically, and is defined for all times $t\in\R$. Every primitive $w$ associated with such a $\lambda$ is supported within the set $\{z(t)\mid t\in \R\}$. The corresponding singular curves $\gamma_{\dot{w}}$ are then supported within the one-dimensional submanifold $\left\{\mathscr{L}(\lambda,t)\mid t\in \R\right\}\subset \G$, where for every $t\in \R$ we have
	\be\label{eq:integration}
		\mathscr{L}(\lambda,t)=\left(\eexp \int_0^t \dot{z}_1(\tau)X_1(\lambda)+\dot{z}_2(\tau)X_2(\lambda)d\tau\right)( e),
	\ee
	and $z(t)$ is as in \eqref{eq:expl_dynamic_maxr} with $z^0=0$. Since the codimension of $\Xi_1$ in $\mathfrak g^*$ is 4, we conclude that $\abn{\Xi_1}=\{\mathscr L(\lambda,t)\mid \lambda\in\Xi_1,\,t\in\R\}$ (compare with Definition~\ref{defi:abnGLambda}) is a sub-analytic set of codimension at least $3$ in $\G$.

	Next, we consider the case of $\lambda\in \Xi_2$ ($\lambda\in \Xi_3$ is analogous). The solution to \eqref{eq:diff_sys_r2_s4} starting at $(0,z_2^0)$ tends to $(0,b(\lambda)_2)$ only as $t\to+\infty$, and we see as well that $z_1(t)\equiv 0$ for all times. Likewise, any curve $z(t)$ in \eqref{eq:expl_dynamic_maxr} starting at $z^0_1\ne 0$ approaches asymptotically the equilibrium $(0,b(\lambda)_2)$ if and only if $z^0_2=b(\lambda)_2$, in which case we conclude that $z_2(t)\equiv b(\lambda)_2$. 
	
	Every primitive $w$ (recall Remark~\ref{rem:wviachangeofcoord}), associated with some $\lambda\in \Xi_2$, is a concatenation of the integral curves of \eqref{eq:diff_sys_r2_s4}. Since we are interested only in those concatenations starting at the origin, we see that all such primitives are supported  within in the set 
	\be
		\ell_\lambda:=\left\{( 0 , t )\mid t\in \R\right\}\cup \left\{(t,  b(\lambda)_2)\mid t\in \R\right\},
	\ee  
	and $w$ may switch between either one of the two components only at the  equilibrium, see Figure~\ref{fig:rank2step4max}.
	
	\begin{figure}
		\includegraphics[scale=.7]{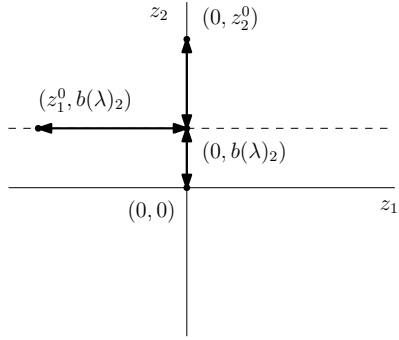}
		\caption{A possible concatenation in $\ell_\lambda$.}
		\label{fig:rank2step4max}
	\end{figure}

	The corresponding singular curves $\gamma_{\dot{w}}$ are then supported within the set $\{\mathscr{L}^1(\lambda,t)\mid t\in \R\}\cup\{\mathscr{L}^2(\lambda,t)\mid t\in \R\}$, where for every $\lambda\in \Xi_2$ and $t\in \R$ we define
	\be
	\begin{aligned}
		 \mathscr{L}^1(\lambda,t):&=\exp(tX_2(\lambda))\in \G\\
		 \mathscr{L}^2(\lambda,t):		 &=\left(\eexp\int_0^1 tX_1(\lambda)d\tau\right)\left( \eexp\int_0^1 b(\lambda)_2X_2(\lambda)d\tau(e)\right)\\
	&=\exp(b(\lambda)_2 X_2(\lambda))\cdot \exp(tX_1(\lambda))\in \G.
	\end{aligned}
	\ee
	Since the codimension of $\Xi_2$ in $\mathfrak g^*$ is at least 4,	 $\abn{\Xi_2}$ is a sub-analytic set of codimension at least $3$ in $\G$. 
	
	\begin{remark}
	One can be more precise in case $\G$ is the free Carnot group of rank 2 and step 4. Indeed, in this case the condition $b(\lambda)_1=0$, which involves the $\mfk g_3^*$ component of $\lambda$,   is  independent from the other requirements on $\lambda$ (i.e. that $\lambda$ has zero projection on $\mfk{g}_1^*\oplus \mfk{g}_2^*$ and that its norm is one). It follows that $\Xi_2$ has codimension at least 5 in $\mfk g^*$, hence $\abn{\Xi_2}$ is a sub-analytic set of codimension at least $4$ in $\G$. 
	
\noindent Similar considerations apply also for the families $\Xi_i$ appearing in the sequel.
	\end{remark}
	
	We summarize the discussion of Case I in the following proposition.
	
	\begin{prop}\label{prop:Sard_detMnegative} For a Carnot group $\G$ of rank $2$ and step $4$, $\abn{\Lambda_1}$ is a sub-analytic set of codimension at least $3$ in $\G$. 
	\end{prop}
	
	\subsection{Case II\texorpdfstring{: $\Lambda=\Lambda_2$}{}} Here $N$ has the form
	\be
		N=\begin{pmatrix} 0 & -1\\
					     1 & 0 \end{pmatrix}.
	\ee

	Given $\lambda\in \Lambda_2$, the integral curves of \eqref{eq:diff_sys_r2_s4} starting at $z^0$ are given by:
	\be\label{eq:expl_dynamic_maxr_detpos}
		\left\{\begin{aligned}
			z_1(t) &= (b(\lambda)_2+z^0_1)\cos t-(z^0_2-b(\lambda)_1)\sin t-b(\lambda)_2,\\
			z_2(t) &= (b(\lambda)_2+z^0_1)\sin t+(z^0_2-b(\lambda)_1)\cos t+b(\lambda)_1;\\
		\end{aligned}\right.
	\ee
namely, these integral curves are circles centered at the equilibrium $(-b(\lambda)_2,b(\lambda)_1)$. In particular, they pass through the equilibrium if and only if $z^0$ is the equilibrium itself, in which case the curves are constant.

We introduce the sets
	\be\begin{aligned}
		\Xi_4&:=\left\{ \lambda\in \Lambda_2\mid b(\lambda)_1= 0,\, b(\lambda)_2= 0  \right\},\\
		\Xi_5&:=\left\{ \lambda\in \Lambda_2\mid b(\lambda)_1\ne 0  \right\},\\
		\Xi_6&:=\left\{ \lambda\in \Lambda_2\mid b(\lambda)_2\ne 0  \right\}.
	\end{aligned}\ee	
	A trajectory $z(t)$ starting at the origin passes through the equilibrium if and only if $\lambda\in \Xi_4$, in which case it stays there for all times. On the other hand, if $\lambda\in \Xi_5$ or $\lambda\in\Xi_6$, $z(t)$ describes a circle through the origin with center in $(-b(\lambda)_2,b(\lambda)_1)$.

	We conclude that every singular curve $\gamma_{\dot{w}}$, associated with a covector $\lambda\in \Xi_4$, reduces to the point $e\in \G$, while the singular curves $\gamma_{\dot{w}}$ associated with covectors $\lambda\in \Xi_5\cup \Xi_6$ are supported within the set $\{\mathscr{L}(\lambda,t)\mid t\in \R\}\subset \G$, where
	$\mathscr{L}(\lambda,t)$ is as in \eqref{eq:integration} and $z(t)$ is as in \eqref{eq:expl_dynamic_maxr_detpos} with $z^0=0$. Since the codimension of $\Lambda_2=\Xi_4\cup\Xi_5\cup\Xi_6$ is 4, we can state the following proposition.

	\begin{prop}\label{prop:Sard_detMpositive} For a Carnot group $\G$ of rank $2$ and step $4$, $\abn{\Lambda_2}$ is a sub-analytic set of codimension at least $3$ in $\G$. 
	\end{prop}

	\subsection{Case III\texorpdfstring{: $\Lambda=\Lambda_3$}{}} Here $N$ has the form
	\be
		N=\begin{pmatrix} 0 & 1\\
					     0 & 0 \end{pmatrix},
	\ee
        and the integral curves of \eqref{eq:diff_sys_r2_s4} starting at $z^0$ are given by:
	\be\label{eq:expl_dynamic_rankone}
		\left\{\begin{aligned}
			z_1(t) &= b(\lambda)_2\frac{t^2}{2}+(b(\lambda)_1+z^0_2)t+z^0_1,\\
			z_2(t) &= b(\lambda)_2t+z^0_2.
		\end{aligned}\right.
	\ee
	
	We observe from the beginning that a necessary condition for the existence of equilibria is that $\lambda\in \Lambda_3\setminus\Xi_7$, where
	\be
		\Xi_7:=\left\{ \lambda\in \Lambda_3\mid b(\lambda)_2\ne 0  \right\}.
	\ee
	In particular, the primitives $w$ associated with a covector $\lambda\in \Xi_7$ are supported within the set  $\{\mathscr{L}(\lambda,t)\mid t\in \R\}\subset \G$, where $\mathscr{L}(\lambda,t)$ is as in \eqref{eq:integration} and $z(t)$ is given by \eqref{eq:expl_dynamic_rankone} with $z^0=0$.
We conclude that $\abn{\Xi_7}$ has codimension at least 3 in $\G$.

	If instead $\lambda\in\Lambda_3\setminus\Xi_7$, i.e. $b(\lambda)_2=0$, equilibria of \eqref{eq:diff_sys_r2_s4} are found on the line $\{(\eta,-b(\lambda)_1)\mid \eta\in \R\}$, and $z(t)$ in \eqref{eq:expl_dynamic_rankone} describes the horizontal line $z_2(t)\equiv z^0_2$. In particular it crosses the set of equilibria if and only if $z^0_2=-b(\lambda)_1$. Recalling that we start with $z^0_2=0$, we consider the sets:
	\be\begin{aligned}
		\Xi_8&:=\left\{ \lambda\in\Lambda_3\mid b(\lambda)_2= 0,\, b(\lambda)_1\ne 0  \right\},\\
		\Xi_9&:=\left\{ \lambda\in\Lambda_3\mid b(\lambda)_2= 0,\, b(\lambda)_1= 0  \right\}.
	\end{aligned}\ee
	
	For every $\lambda\in \Xi_8$, the primitives $w$ associated with $\lambda$ are supported within the horizontal axis $\ell\subset \R^2$, and 
	$\abn{\Xi_8}$ is a sub-analytic set of codimension at least $3$ in $\G$. 	
	Similar conclusions hold for $\lambda\in \Xi_9$, because  in this case the primitives $w$ are supported within the horizontal axis, which coincides here with the set of equilibria.
	
	\begin{remark}
	We observe that $\Xi_9$ has codimension at least 5 in $\mfk g^*$: indeed, the condition $b(\lambda)=0$ necessarily imposes at least one extra condition on the $\mfk g_3^*$ component of $\lambda$, for otherwise one would get $\mfk g_3^*=0$.  In particular, we have the better lower bound 4 on the codimension of $\abn{\Xi_9}$.
	\end{remark}
	
	\begin{prop}\label{prop:Sard_rankMone} For a Carnot group $\G$ of rank $2$ and step $4$, $\abn{\Lambda_3}$ is a sub-analytic set of codimension at least $3$ in $\G$.
	\end{prop}
	
	\begin{remark}\label{rem:GoleKaridi}
	C.~Gol\'e and R.~Karidi provided in \cite{GoleKaridi}   examples of strictly 
	singular length minimizing curves. One of their examples is revisited  in \cite[Section 6.3]{LLMV_GAFA}: this example is a parabola-type curve as in~\eqref{eq:expl_dynamic_rankone}  associated with some $\lambda\in\Lambda_3$.
\end{remark}	
	
	\subsection{Case IV\texorpdfstring{: $\Lambda=\Lambda_4$}{}} The condition $M(\lambda)=0$ implies that the projection of $\lambda$ onto $\mfk{g}_4^*$ is zero, and this implies that $v(\lambda)\ne 0$, for otherwise the covector $\lambda$ itself would be zero. Solutions to \eqref{eq:diff_sys_r2_s4} are therefore parallel lines and the concatenations giving the possible primitives $w$ are simply lines through the origin. We conclude as before.
	
	\begin{prop}\label{prop:Sard_Mzero} For a Carnot group $\G$ of rank $2$ and step $4$, $\abn{\Lambda_4}$ is a sub-analytic set of codimension at least $3$ in $\G$.
	\end{prop}
	
	The proof of Theorem~\ref{thm:S_r2_s4} is complete.

	\section{Carnot groups of rank 3 and step 3}\label{sec:r3s3}
	
	Consider a Carnot group $\G$ of rank $3$ and step $3$, and pick a singular trajectory $\gamma_u\subset \G$. Let $u\in L^1([0,1],\R^3)$ be the control associated with $\gamma_u$. By Proposition~\ref{prop:singular}, the elements of the form 
	\be\label{eq:subspace_non_surj_s3r3}
		[X_{w(t)},X_j](e)+\int_0^t[X_{w(\tau_1)},[X_{u(\tau_1)},X_j]](e),\ \ t\in [0,1], \ \ j=1,2,3
	\ee
	do not generate the subspace $\mfk{g}_2\oplus \mfk{g}_3$, and therefore, up to differentiating \eqref{eq:subspace_non_surj_s3r3}, one gets the existence of a
	covector $\lambda\in \mathbb{S}(\mfk{g}_2^*\oplus \mfk{g}_3^*)$
	such that 
	\be
		\lambda_{u(t)j}+\lambda_{w(t)u(t)j}=0,
	\ee
	for $j=1,2,3$ and a.e. $t\in [0,1]$. 
	
	We introduce the skew-symmetric matrix $\mathscr{M}_u(\lambda,t)\in M_3(\R)$ defining
	\be
		\mathscr{M}_u(\lambda,t)_{ij}=\lambda_{ij}+\lambda_{w(t)ij},\ \ 1\le i,j\le 3.
	\ee	
	Then $u\in \ker(\mathscr{M}_u(\lambda,t))$ for a.e. $t\in [0,1]$ by Proposition~\ref{prop:criterion_abnormal}.
	
	Let $I_{\max}\subset [0,1]$ be a maximal open set where $\rank(\mathscr{M}_u(\lambda,t))=2$, and observe that  $\mathscr{M}_u(\lambda,t))$ is zero on the complement $[0,1]\setminus I_{\max}$. For a.e. $t\in I_{\max}$, $u$ is parallel to 
	\be
		\begin{pmatrix}
			\lambda_{23}+\lambda_{w(t)23} \\
			\lambda_{31}+\lambda_{w(t)31} \\
			\lambda_{12}+\lambda_{w(t)12}
		\end{pmatrix}.
	\ee 
	As in the previous section, we drop the parametrization of $\gamma_u$, and we see that all the primitives $w$
	such that $\lambda\in \IM(G_e^{\dot{w}})^\perp$ are obtained by concatenation of  the integral curves of the differential system 
	\be\label{eq:dif_sys_r3_s3}
		\dot{x}(t)=M(\lambda)x(t)+v(\lambda),\  \ x\in \R^3,          
	\ee
	where we defined
	\be
		\begin{aligned}
			M:\mathbb{S}(\mfk{g}_2^*\oplus \mfk{g}_3^*)&\to M_3(\R),\\
			\lambda&\mapsto \begin{pmatrix} \lambda_{123} &  \lambda_{223} &  \lambda_{323} \\  \lambda_{131} &  \lambda_{231} &  \lambda_{331} \\  \lambda_{112} &  \lambda_{212} &  \lambda_{312}  \end{pmatrix},
		\end{aligned}\ \ \ \ \ \ 
		\begin{aligned}
			v:\mathbb{S}(\mfk{g}_2^*\oplus \mfk{g}_3^*)&\to \R^3,\\
			\lambda&\mapsto \begin{pmatrix}  \lambda_{23} \\  \lambda_{31} \\  \lambda_{12} \end{pmatrix}.
		\end{aligned}
	\ee
	Again, since the  primitives $w$ satisfy $w(0)=0$, one has to take into account only those concatenations starting at the origin. 
	Observe that, as a consequence of Jacobi's identity, the matrix $M(\lambda)$ has zero trace. Notice moreover that the set $[0,1]\setminus I_{\max}$ coincides with the set of times $t\in [0,1]$ such that the solution $x(t)$ to \eqref{eq:dif_sys_r3_s3} crosses the set of equilibria  of the system.
	
	Keeping track of the zero-trace condition on $M(\lambda)$, we stratify $\mathbb{S}(\mfk{g}_2^*\oplus \mfk{g}_3^*)$ as follows:
	\be\label{eq:stratification_r3_s3}
		\begin{aligned}
			\Lambda_1&:=\left\{ \lambda\in \mathbb{S}(\mfk{g}_2^*\oplus \mfk{g}_3^*)\mid \det(M(\lambda))\ne 0,\, M(\lambda)\;\textrm{has three distinct real eigenvalues} \right\},\\
			\Lambda_2&:=\left\{ \lambda\in \mathbb{S}(\mfk{g}_2^*\oplus \mfk{g}_3^*)\mid \det(M(\lambda))\ne 0,\, M(\lambda)\;\textrm{has two distinct real eigenvalues} \right\},\\
			\Lambda_3&:=\left\{ \lambda\in \mathbb{S}(\mfk{g}_2^*\oplus \mfk{g}_3^*)\mid \det(M(\lambda))\ne 0,\, M(\lambda)\;\textrm{has two non-real eigenvalues} \right\},\\
			\Lambda_4&:=\left\{ \lambda\in \mathbb{S}(\mfk{g}_2^*\oplus \mfk{g}_3^*)\mid \det(M(\lambda))\ne 0,\, M(\lambda)\;\textrm{has a generalized eigenvector of order } 2 \right\},\\
			\Lambda_5&:=\left\{ \lambda\in \mathbb{S}(\mfk{g}_2^*\oplus \mfk{g}_3^*)\mid \rank(M(\lambda))=2,\, M(\lambda)\;\textrm{has two real eigenvalues} \right\},\\
			\Lambda_6&:=\left\{ \lambda\in \mathbb{S}(\mfk{g}_2^*\oplus \mfk{g}_3^*)\mid \rank(M(\lambda))=2,\, M(\lambda)\;\textrm{has two non-real eigenvalues}  \right\},\\
			\Lambda_7&:=\left\{ \lambda\in \mathbb{S}(\mfk{g}_2^*\oplus \mfk{g}_3^*)\mid \rank(M(\lambda))=2,\, M(\lambda)\;\textrm{has a generalized eigenvector of order } 3 \right\},\\
			\Lambda_8&:=\left\{ \lambda\in \mathbb{S}(\mfk{g}_2^*\oplus \mfk{g}_3^*)\mid \rank(M(\lambda))=1\right\},\\
			\Lambda_9&:=\left\{ \lambda\in \mathbb{S}(\mfk{g}_2^*\oplus \mfk{g}_3^*)\mid M(\lambda)=0 \right\}.
		\end{aligned}
	\ee

	It is again convenient to change coordinates: we assume that $M(\lambda)$ is in its normal form $N=N(\lambda)$ and we complete the proof of Theorem~\ref{thm:S_r3_s3} analyzing separately each possibility for $N$. Recall that, similarly to \eqref{eq:first_norfor}, the change of coordinates $\lambda\mapsto P(\lambda)$ and $\lambda\mapsto P(\lambda)^{-1}$ can be chosen to be sub-analytic on each of the sets $\Lambda_i$. Then we write
	\be\label{eq:diff_sys_r3_s3}
		\dot{z}=Nz+b(\lambda),\ \ z\in \R^3,
	\ee
	with the same conventions as in \eqref{eq:diff_sys_r2_s4}. We recall that this choice of coordinates induces a sub-analytic change of the frame $\lambda\mapsto \left(X_1(\lambda),X_2(\lambda),X_3(\lambda)\right)$ as in Remark~\ref{rem:changeofcoord}. We also make an abuse of notation similarly to Remark~\ref{rem:wviachangeofcoord}, identifying primitives $w$ with their new coordinate presentation $P(\lambda)^{-1}w$.

	\subsection{Case I\texorpdfstring{: $\Lambda=\Lambda_1$}{}}\label{sec:r3_s3_diag} Here
	\be
		N=\begin{pmatrix}
			a & 0 & 0 \\
			0 & b & 0 \\
			0 & 0 & -(a+b) 
		\end{pmatrix}, \ \ a,b\in \R\setminus\{0\},\ \ ab>0, \ \ |a|>|b|.
	\ee
The solution to \eqref{eq:diff_sys_r3_s3} starting from the point $z^0$ is given by:
	\be\label{eq:traj:s3_r3_max}\left\{
		\begin{aligned}
			z_1(t) &= \frac{e^{at}-1}{a}b(\lambda)_1+e^{at}z^0_1,\\
			z_2(t) &= \frac{e^{bt}-1}{b}b(\lambda)_2+e^{bt}z^0_2,\\			
			z_3(t) &= -\frac{e^{-(a+b)t}-1}{a+b}b(\lambda)_3+e^{-(a+b)t}z^0_3,
		\end{aligned}
		\right.
	\ee
	and the equilibrium set reduces to the single point $\left(-\frac{b(\lambda)_1}{a},-\frac{b(\lambda)_2}{b},\frac{b(\lambda)_3}{a+b}\right)$.
	
	The curve $z(t)$ tends  to the equilibrium (either as $t\to+\infty$ or as $t\to-\infty$) if and only if 
	\be\label{eq:Z_dif_sys_r3_s3}
\text{either $\frac{b(\lambda)_1}{a}+z^0_1=\frac{b(\lambda)_2}{b}+z^0_2=0$,\ \ or  $-\frac{b(\lambda)_3}{a+b}+z^0_3=0$.}
	\ee
It is not restrictive to discuss the cases in which only one of  these conditions holds (if both conditions hold, $z(t)$ is constant). Assuming for example $-\frac{b(\lambda)_3}{a+b}+z^0_3=0$ and $a> b>0$, then
	\be
	\lim_{t\to +\infty}\left( z_1(t)^2+z_2(t)^2\right)=+\infty;
	\ee
	(this limit tends to $+\infty$ as well for $t\to -\infty$ if $a< b<0$). We conclude that $z(t)$ tends asymptotically to the equilibrium only once, either as $t\to +\infty$ or as $t\to -\infty$. 	
	
	Recalling that we are interested only in concatenations of solutions to  \eqref{eq:diff_sys_r3_s3} starting from the origin, we introduce the sets
	\begin{align}
		\Xi_1&:=\left\{ \lambda\in \Lambda_1\mid b(\lambda)_1^2+b(\lambda)_2^2\ne 0, \, b(\lambda)_3\ne 0  \right\},\\
		\Xi_2&:=\left\{ \lambda\in \Lambda_1\mid b(\lambda)_3= 0 \right\},\\
		\Xi_3&:=\left\{ \lambda\in \Lambda_1\mid b(\lambda)_1=b(\lambda)_2= 0  \right\}.
	\end{align}
	
	\begin{figure}
		\includegraphics[scale=.6]{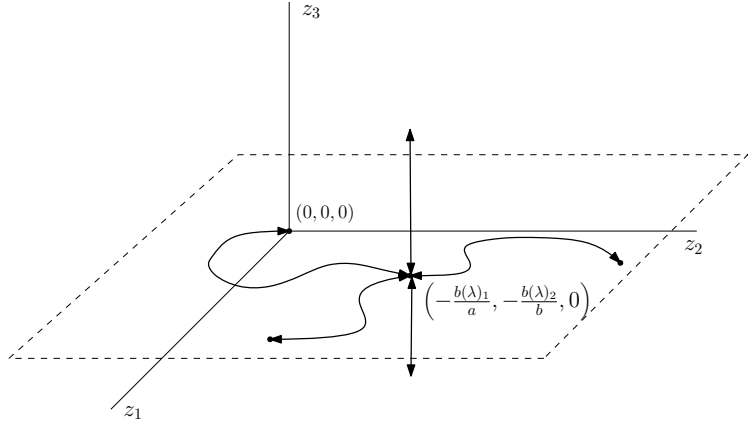}
		\caption{The reachable set from the origin by trajectories of \eqref{eq:diff_sys_r3_s3} when $\lambda\in \Xi_3$.}
		\label{fig:rank3step3max}
	\end{figure}
	
	For every $\lambda\in \Xi_1$, the trajectory $z(t)$ with $z^0=0$ never approaches the equilibrium. All the primitives $w$ associated with such values of $\lambda$ are supported within $\{z(t)\mid t\in \R\}$, and the corresponding singular curves $\gamma_{\dot{w}}$ are supported within $\{\mathscr{L}(\lambda,t)\mid\lambda\in\Xi_1,t\in\R\}$, where
	\be\label{eq:ricostruzione_w}
		\mathscr{L}(\lambda,t):=\left(\eexp \int_0^t \dot z_1(\tau) X_1(\lambda)+\dot z_2(\tau)X_2(\lambda)+\dot z_3(\tau)X_3(\lambda)d\tau\right)( e)
	\ee
and $z(t)$ is as in \eqref{eq:traj:s3_r3_max} with $z^0=0$.
 Taking into account that $\lambda$ has zero $\mfk{g}_1^*$ component, we deduce that $\abn{\Xi_1}$ is a sub-analytic set of codimension at least $3$ in $\G$.
 
\begin{remark}\label{rem:ex1}
We observe that the 
singular curve discussed in \cite[Section 6.3]{Vit_Sard} is associated with a covector $\lambda\in\Xi_1$.
\end{remark}
 
 If instead $\lambda\in \Xi_2$, the equilibrium point is $(\frac{-b(\lambda)_1}{a},\frac{-b(\lambda)_2}{b},0)$ and the trajectory $z(t)$ in \eqref{eq:traj:s3_r3_max} can approach the equilibrium only if either $z_1^0+\frac{b(\lambda)_1}{a}=z_2^0+\frac{b(\lambda)_2}{b}=0$ or  $z_3^0=0$. Any primitive $w$ starting at the origin and associated with $\lambda\in\Xi_2$, which is a concatenation of trajectories of \eqref{eq:diff_sys_r3_s3}, must then be initially supported within the curve $z(t)$ in \eqref{eq:traj:s3_r3_max} with $z^0=0$, i.e., in the plane $z_3=0$, until it approaches the equilibrium. The point in $\G$ corresponding to the equilibrium is then
 \be
  g_0:=\left(\eexp \int_0^{-\infty} \dot z_1(\tau) X_1(\lambda)+\dot z_2(\tau)X_2(\lambda)d\tau\right)( e),
 \ee
 where $z(t)$ is the trajectory in \eqref{eq:traj:s3_r3_max} with $z^0=0$, and the chronological exponential above is to be intended with the same meaning as in Remark~\ref{rem:infinite}. 
 From the equilibrium, it can then continue either by flowing along trajectories supported within the same plane, or by following the line $z_1+\frac{b(\lambda)_1}{a}=z_2+\frac{b(\lambda)_2}{b}=0$, and by switching between these two possibilities at the equilibrium, potentially infinitely many times (see Figure~\ref{fig:rank3step3max}). Any such primitive  is then supported within the set 
 \be
\left\{g_0\cdot\exp(tX_3(\lambda))\mid t\in\R\right\}\cup\{\mathscr L(\lambda,\theta,t)\mid t\in\R,\lambda\in\Xi_2,\theta\in\mathbb S^1\},
 \ee
where, setting $z^\theta(t)$ as the trajectory $z(t)$ as in \eqref{eq:traj:s3_r3_max} with $z^0=(\cos\theta-\frac{b(\lambda)_1}{a},\sin\theta-\frac{b(\lambda)_2}{b},0)$, we defined 
\be\label{eq:set_2touse}
\mathscr L(\lambda,\theta,t):=\left(\eexp \int_{-\infty}^t \dot z_1^\theta(\tau) X_1(\lambda)+\dot z_2^\theta(\tau)X_2(\lambda)d\tau\right)(g_0)
\ee
Again, if $a<b<0$ the chronological exponential above should be taken from $+\infty$ to $t$ (with limits of integration in this order, see Remark~\ref{rem:infinite}). Since $\Xi_2$ has codimension at least 4 in $\mfk g^*$, we conclude that $\abn{\Xi_2}$ is a sub-analytic set of codimension at least $2$ in $\G$ (see Figure~\ref{fig:rank3step3max}).

 If $\lambda\in \Xi_3$, the equilibrium point is $(0,0,\frac{b(\lambda)_3}{a+b})$. It can be easily checked that the trajectory $z(t)$ in \eqref{eq:traj:s3_r3_max} can approach the equilibrium only if either $z_1^0=z_2^0=0$ or $z_3^0=\frac{b(\lambda)_3}{a+b}$. Any primitive $w$ starting at the origin and associated with $\lambda\in\Xi_3$ is then supported within the set 
 \be\label{eq:sets_1}
\{\exp(tX_3(\lambda))\mid t\in\R\}\cup\{\mathscr L(\lambda,\theta,t)\mid t\in\R,\lambda\in\Xi_3,\theta\in\mathbb S^1\},
 \ee
where, setting $z^\theta(t)$ as the trajectory $z(t)$ as in \eqref{eq:traj:s3_r3_max} with $z^0=(\cos\theta,\sin\theta,\frac{b(\lambda)_3}{a+b})$, we defined
\be 
\mathscr L(\lambda,\theta,t):=\left(\eexp \int_{-\infty}^t \dot z_1^\theta(\tau) X_1(\lambda)+\dot z_2^\theta(\tau)X_2(\lambda)d\tau\right)\left(\exp\left(\frac{b(\lambda)_3}{a+b}X_3(\lambda)\right)\right).
\ee
Again, if $a<b<0$ the chronological exponential above should be taken from $+\infty$ to $t$ (in this order). We conclude that $\abn{\Xi_3}$ is a sub-analytic set of codimension at least $2$ in $\G$.

\begin{remark}\label{rem:remark_free}
When $\G$ is the free group of rank 3 and step 3 one can be more precise: indeed, $\Xi_2$ and $\Xi_3$ have higher codimension and it follows that $\abn{\Xi_2}$ and $\abn{\Xi_3}$ are sub-analytic sets of codimension  $3$ and $4$ in $\G$, respectively. 
\end{remark}

The discussion of Case I can be summarized as follows.

\begin{prop}\label{prop:Sard_r3_s3_Mdiag}
	For a Carnot group $\G$ of rank $3$ and step $3$, $\abn{\Lambda_1}$ is a sub-analytic set of codimension at least $2$ in $\G$.
\end{prop}

	\subsection{Case II\texorpdfstring{: $\Lambda=\Lambda_2$}{}} One can treat this case exactly as the case $\Lambda=\Lambda_1$ with $a=b$;  distinguishing the two cases is necessary to guarantee that the change of coordinates $\lambda\mapsto P(\lambda)$ is sub-analytic for $\lambda\in\Lambda_i$, $i=1,2$. 
	
	\begin{prop}\label{prop:Sard_r3_s3_Mdiageqeigen} For a Carnot group $\G$ of rank $3$ and step $3$, $\abn{\Lambda_2}$ is a sub-analytic set of codimension at least $2$ in $\G$.
     \end{prop}
	
	\subsection{Case III\texorpdfstring{: $\Lambda=\Lambda_3$}{}} Here we have
	\be
		N=\begin{pmatrix}
			1 & -a & 0 \\
			a & 1 & 0 \\
			0 & 0 & -2
		\end{pmatrix},\ \ a\in\R\setminus\{0\}.
	\ee
	Setting
	\be
		\alpha:=b(\lambda)_1+ab(\lambda)_2+(1+a^2)z_1^0 \ \ \text{and}\ \ \beta:=-ab(\lambda)_1+b(\lambda)_2+(1+a^2)z_2^0,
	\ee
	the solution to \eqref{eq:diff_sys_r3_s3} starting from the point $z^0$ is given by:
	\be\label{eq:Z_dif_sys_r3_s3_cplx}
		\left\{\begin{aligned}
			z_1(t) = & \frac{e^t}{a^2+1}\left(\frac{\alpha-i\beta}{2}e^{-iat}+\frac{\alpha+i\beta}{2}e^{iat}\right)-\frac{b(\lambda)_1+ab(\lambda)_2}{a^2+1}\\
			z_2(t) = & \frac{e^t}{a^2+1}\left(\frac{\alpha-i\beta}{2}e^{iat}+\frac{\alpha+i\beta}{2}e^{-iat}\right)-\frac{-ab(\lambda)_1+b(\lambda)_2}{a^2+1}\\
			z_3(t)  = & \frac{1}{2}e^{-2t}\left( -b(\lambda)_3+2z_3^0 \right)+\frac{b(\lambda)_3}2.
		\end{aligned}\right.
	\ee
	A trajectory $z(t)$ passes through the equilibrium $\left(-\frac{b(\lambda)_1+ab(\lambda)_2}{1+a^2},-\frac{-ab(\lambda)_1+b(\lambda)_2}{1+a^2},\frac{b(\lambda)_3}{2}\right)$ if and only if
	\be
\text{either}\ \ \frac{b(\lambda)_1+ab(\lambda)_2}{1+a^2}+z_1^0=\frac{-ab(\lambda)_1+b(\lambda)_2}{1+a^2}+z_2^0=0\ \ \text{or}\ \ 2z_3^0-b(\lambda)_3=0.
	\ee
	Recalling that we are interested only in concatenations of solutions to  \eqref{eq:diff_sys_r3_s3} starting from the origin, we introduce the sets
	\begin{align}
		\Xi_4&:=\left\{ \lambda\in \Lambda_3\mid b(\lambda)_1^2+b(\lambda)_2^2\ne 0, \, b(\lambda)_3\ne 0  \right\},\\
		\Xi_5&:=\left\{ \lambda\in \Lambda_3\mid b(\lambda)_1=b(\lambda)_2= 0  \right\},\\
		\Xi_6&:=\left\{ \lambda\in \Lambda_3\mid b(\lambda)_3= 0 \right\}.
	\end{align}
	It is clear at this point that, for every $\lambda\in \Xi_4$, the singular curves $\gamma_{\dot{w}}$ associated with $\lambda$ are supported within a one-dimensional submanifold $\{\mathscr{L}(\lambda,t)\mid t\in\R\}$, where $\mathscr L$ is defined as in \eqref{eq:ricostruzione_w}, 
	so that $\abn{\Xi_4}$ is a sub-analytic set of codimension at least $3$ in $\G$. 
	
	If instead $\lambda\in \Xi_5$, a trajectory $z(t)$ approaches the equilibrium $\left(0,0,\frac{b(\lambda)_3}{2}\right)$ only if either $z_1^0=z_2^0=0$ or $z_3^0=\frac{b(\lambda)_3}{2}$. Any primitive $w$ associated with $\lambda\in\Xi_5$ is then supported within the set 
 \be
	\{\exp(tX_3(\lambda))\mid t\in\R\}\cup\{\mathscr L(\lambda,\theta,t)\mid t\in\R,\lambda\in\Xi_2,\theta\in\mathbb S^1\},
 \ee
where
 \be
 \mathscr L(\lambda,\theta,t):=\left(\eexp \int_{-\infty}^t \dot z_1^\theta(\tau) X_1(\lambda)+\dot z_2^\theta(\tau)X_2(\lambda)d\tau\right)\left(\exp\left(\frac{b(\lambda)_3}{2}X_3(\lambda)\right)\right).
 \ee 
 and $z^\theta(t)$ is the trajectory $z(t)$ as in \eqref{eq:Z_dif_sys_r3_s3_cplx} with $z^0=(\cos\theta,\sin\theta,\frac{b(\lambda)_3}{2})$. We conclude that $\abn{\Xi_5}$ is a sub-analytic set of codimension at least $2$ in $\G$.
 	
If $\lambda\in \Xi_6$, a curve $z(t)$ as in \eqref{eq:Z_dif_sys_r3_s3_cplx} reaches the equilibrium $\left(-\frac{b(\lambda)_1+ab(\lambda)_2}{1+a^2},-\frac{-ab(\lambda)_1+b(\lambda)_2}{1+a^2},0\right)$ only if 
\be
\text{either}\ \ z_3^0=0\ \ \text{or}\ \ z_1^0+\frac{b(\lambda)_1+ab(\lambda)_2}{1+a^2}=z_2^0+\frac{-ab(\lambda)_1+b(\lambda)_2}{1+a^2}=0.
\ee
Any primitive $w$ associated with $\lambda\in\Xi_4$ is then supported within the union of the sets 
 \be
\left\{g_0\cdot\exp(tX_3(\lambda)\mid t\in\R\right\}\cup \{\mathscr L(\lambda,\theta,t)\mid t\in\R,\lambda\in\Xi_6,\theta\in\mathbb S^1\}
 \ee
 where, if $z(t)$ is as in \eqref{eq:Z_dif_sys_r3_s3_cplx} with $z^0=0$, we defined
  \be
  g_0:=\left(\eexp \int_0^{-\infty} \dot z_1(\tau) X_1(\lambda)+\dot z_2(\tau)X_2(\lambda)d\tau\right)( e)
 \ee
 and  $\mathscr L(\lambda,\theta,t)$ is given by \eqref{eq:set_2touse}, provided  $z^\theta(t)$ is the trajectory $z(t)$ as in  \eqref{eq:Z_dif_sys_r3_s3_cplx} with $z^0=\left( \cos\theta -\frac{b(\lambda)_1+ab(\lambda)_2}{1+a^2}, \sin\theta -\frac{-ab(\lambda)_1+b(\lambda)_2}{1+a^2},0 \right)$. We deduce that $\abn{\Xi_6}$ is a sub-analytic set of codimension at least $2$ in $\G$.

	\begin{prop}\label{prop:Sard_r3_s3_Mcomplex} For a Carnot group $\G$ of rank $3$ and step $3$, $\abn{\Lambda_3}$ is a sub-analytic set of codimension at least $2$ in $\G$.
			\end{prop}

	\subsection{Case IV\texorpdfstring{: $\Lambda=\Lambda_4$}{}} Here
	\be
		N=\begin{pmatrix}
				1 & 1 & 0 \\
				0 & 1 & 0 \\
				0 & 0 & -2
	             \end{pmatrix}.
	\ee
The solution to \eqref{eq:diff_sys_r3_s3} starting from the point $z^0$ is given by
	\be
		\left\{
			\begin{aligned}
				z_1(t) &= (e^t-1)b(\lambda)_1+e^t(t-1)b(\lambda)_2+b(\lambda)_2+e^t(z_1^0+tz_2^0),\\
				z_2(t) &= (e^t-1)b(\lambda)_2+e^{t}z^0_2,\\			
				z_3(t) &= -\frac{e^{-2t}-1}{2}b(\lambda)_3+e^{-2t}z_3^0,
			\end{aligned}
		\right..
	\ee
	A curve $z(t)$ passes through the equilibrium $\left(-b(\lambda)_1+b(\lambda)_2,-b(\lambda)_2,\frac{b(\lambda)_3}{2}\right)$ if and only if either $b(\lambda)_2+z_2^0=b(\lambda)_1-b(\lambda)_2+z_1^0=0$ or $-\frac{b(\lambda)_3}{2}+z_3^0=0$. 
	The situation is similar to that of Case I in Section~\ref{sec:r3_s3_diag}, and the computations are left to the reader.

	\begin{prop}\label{prop:Sard_r3_s3_MJord}
		For a Carnot group $\G$ of rank $3$ and step $3$, $\abn{\Lambda_4}$ is a sub-analytic set of codimension at least $2$ in $\G$.
	\end{prop}
		
	\subsection{Case V\texorpdfstring{: $\Lambda=\Lambda_5$}{}} Here
	\be
		N=\begin{pmatrix}
				1 & 0 & 0 \\
				0 & -1 & 0 \\
				0 & 0 & 0
	             \end{pmatrix},
	\ee
	and the solution $z(t)$ to \eqref{eq:diff_sys_r3_s3} starting from the point $z^0$ is given by
	\be
		\left\{
			\begin{aligned}
				z_1(t) &= (e^t-1)b(\lambda)_1+e^{t}z_1^0,\\
				z_2(t) &= -(e^{-t}-1)b(\lambda)_2+e^{-t}z_2^0,\\			
				z_3(t) &= b(\lambda)_3t+z_3^0.
			\end{aligned}
		\right.
	\ee
	A necessary condition for the existence of equilibria is that $\lambda\in \Lambda_5\setminus \Xi_7$, where
	\be
		\Xi_7:=\left\{\lambda\in \Lambda_5\mid b(\lambda)_3\ne 0\right\}.
	\ee
	Therefore, 	$\abn{\Xi_7}$ is a sub-analytic set of codimension at least $3$ in $\G$.
	
	If $\lambda\in \Lambda_5\setminus \Xi_7$, equilibria of \eqref{eq:diff_sys_r3_s3} are found on the line $\left\{(-b(\lambda)_1,b(\lambda)_2,\eta)\mid \eta \in \R\right\}$, and a curve $z(t)$ approaches the equilibrium set if and only if
	\be\label{eq:equilibrium_conditions}
		\text{either $b(\lambda)_1+z_1^0=0$ or $-b(\lambda)_2+z^0_2=0$. }
	\ee
	Recalling that we are interested only in concatenations of solutions to  \eqref{eq:diff_sys_r3_s3} starting from the origin, we introduce the sets
	\be\begin{aligned}
		\Xi_{8}&:=\left\{\lambda\in \Lambda_5\mid b(\lambda)_3= 0,\, b(\lambda)_1\ne 0,\, b(\lambda)_2\ne 0\right\},\\
		\Xi_{9}&:=\left\{\lambda\in \Lambda_5\mid b(\lambda)_3= 0,\, b(\lambda)_1= 0\right\},\\
		\Xi_{10}&:=\left\{\lambda\in \Lambda_5\mid b(\lambda)_3= 0,\, b(\lambda)_2= 0\right\}.
	\end{aligned}\ee
		Since $z(t)$ does not approach the set of equilibria if $\lambda\in \Xi_8$, we readily deduce that $\abn{\Xi_8}$ is a sub-analytic set of codimension at least $3$ in $\G$. 	
	
	 The cases $\lambda\in \Xi_{9}$ and $\lambda\in \Xi_{10}$ are symmetric, and without loss of generality we study only the first one. Assume then that $\lambda\in \Xi_{9}$. 
	\begin{figure}
		\includegraphics[scale=.7]{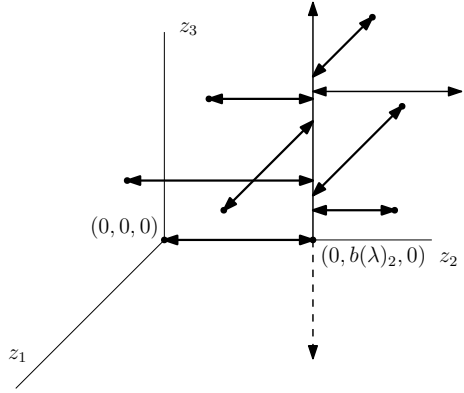}
		\caption{Concatenation of solutions to \eqref{eq:diff_sys_r3_s3} if $\lambda\in \Xi_{9}$.}
		\label{fig:rank3step3rank2}
	\end{figure}
	The trajectory $z(t)$ through $z^0$ is now given by
	\be
		\left\{
			\begin{aligned}
				z_1(t) &= e^{t}z_1^0,\\
				z_2(t) &= -(e^{-t}-1)b(\lambda)_2+e^{-t}z_2^0,\\			
				z_3(t) &= z_3^0.
			\end{aligned}
		\right.
	\ee
	In particular we see that if \eqref{eq:equilibrium_conditions} is satisfied, then either $z_1(t)\equiv 0$, or $z_2(t)\equiv b(\lambda)_2$. Since the $w_3$ coordinate of a primitive is allowed to change only within the line  
	$\left\{(0,b(\lambda)_2,\eta)\mid \eta \in \R\right\}$ of equilibria, we conclude that 	the primitives $w$ associated with $\lambda\in\Xi_9$ are supported within the union $\Pi(\lambda)$ of the planes $\left\{(0,z_2,z_3)\mid (z_2,z_3)\in \R^2\right\}$ and $\left\{(z_1,b(\lambda)_2,z_3)\mid (z_1,z_3)\in \R^2\right\}$, as shown in Figure~\ref{fig:rank3step3rank2}.
	Observe that the concatenations of solutions $z(t)$ have a ``tree-like'' structure within $\Pi(\lambda)$. Any primitive $w$ associated with $\lambda\in\Xi_9$ is supported within the set $\{\mathscr L(\lambda,z^0)\mid \lambda\in\Xi_9,z^0\in \Pi(\lambda)\}$ where, denoting by $\ell^{z^0}:[0,1]\to \R^3$ the unique injective absolutely continuous curve joining the origin and $z^0$, and realized as a concatenation of solutions to \eqref{eq:diff_sys_r3_s3}, we defined
\be
\mathscr L(\lambda,z^0):=\left(\eexp \int_0^1 \dot \ell^{z^0}_1(\tau) X_1(\lambda)+\dot \ell^{z^0}_2(\tau)X_2(\lambda)+\dot \ell^{z^0}_3(\tau)X_3(\lambda)d\tau\right)(e).
\ee
In particular, we deduce that $\abn{\Xi_9}$ is a sub-analytic set of codimension at least $2$ in $\G$.

	\begin{prop}
	For a Carnot group $\G$ of rank $3$ and step $3$, $\abn{\Lambda_5}$ is a sub-analytic set of codimension at least $2$ in $\G$.
	\end{prop}

		\begin{remark}\label{rem:ex2}
	Let $\mathbb F$ be the free  Carnot group of rank 3 and step 3. The following curve $\gamma:[0,1]\to\mathbb F$, with associated control $u\in L^1([0,1],\R^3)$, starts from the origin and sweeps three segments on the coordinate axes:
	\be
		u(t)=
		\begin{cases}
		(1,0,0) &\text{if }t\in[0,1/6)\\
		(-1,0,0) &\text{if }t\in[1/6,2/6)\\
		(0,1,0) &\text{if }t\in[2/6,3/6)\\
		(0,-1,0) &\text{if }t\in[3/6,4/6)\\
		(0,0,1) &\text{if }t\in[4/6,5/6)\\
		(0,0,-1) &\text{if }t\in[5/6,1],
		\end{cases}\qquad
		\gamma(t)=
		\begin{cases}
		\exp(tX_1) &\text{if }t\in[0,1/6]\\
		\exp((2/6-t)X_1) &\text{if }t\in[1/6,2/6]\\
		\exp((t-2/6)X_2) &\text{if }t\in[2/6,3/6]\\
		\exp((4/6-t)X_2) &\text{if }t\in[3/6,4/6]\\
		\exp((t-2/6)X_3) &\text{if }t\in[4/6,5/6]\\
		\exp((1-t)X_3) &\text{if }t\in[5/6,1].
		\end{cases}
	\ee
	Let us check that $\gamma$ is singular. We use Proposition~\ref{prop:singular} and compute the subspace $\mfk{R}_u$ in~\eqref{eq:subsp_A}:	
	\begin{align*}
		\mfk{R}_u&=\Span_{Y\in \mfk{g}_1,t\in [0,1]}\left\{
		\int_0^t[X_{u(\tau)},Y]d\tau + \int_0^t\int_0^\tau[X_{u(\sigma)},[X_{u(\tau)},Y]]d\sigma d\tau\right\}\\
		&=\Span_{Y\in \mfk{g}_1,t\in [0,1]}\left\{
		[X_{w(\tau)},Y] + \int_0^t[X_{w(\tau)},[X_{u(\tau)},Y]] d\tau\right\},\\
		\intertext{where $w$ is as usual the primitive of $u$. Since, for all $\tau$, $w(\tau)$ and $u(\tau)$ are parallel, and actually of the form $(sX_i,\pm X_i)$ for some $s=s(\tau)\in\R$ and $i\in\{1,2,3\}$, we deduce}
		\mfk{R}_u&\subset \mfk g_2\oplus \Span\{X_{iij}\mid i,j\in\{1,2,3\}\}.
	\end{align*}
	In particular, $\mfk R_u$ is a proper subspace of $ \mfk g_2\oplus\mfk g_3$ because $X_{123}$ and $ X_{213}$ do not belong to $\mfk R_u$.	
	
	The singular curve $\gamma$ is associated	with a covector $\lambda\in\Xi_9\cap\Xi_{10}=\{\lambda\in\Lambda_5\mid b(\lambda)=0\}$ and is associated with a dynamical system with the equilibrium point at the origin; compare with Figure~\ref{fig:rank3step3rank2}. Actually, one can choose $\lambda$ in such a way that $\lambda\in\mfk g_2^\perp$ (i.e., $v(\lambda)=b(\lambda)=0$), $\lambda_{iij}=0$ for all couples $i,j$ and $\lambda_{123}=1,\lambda_{231}=-1$. We also observe that such a $\gamma$ provides a new example of a singular curve that is not contained in any subgroup of $\mathbb F$, see \cite[Section 6.3]{Vit_Sard}. 
	
	\end{remark}

	\subsection{Case VI\texorpdfstring{: $\Lambda=\Lambda_6$}{}} Here
	\be
		N=\begin{pmatrix}
				0 & -1 & 0 \\
				1 & 0 & 0 \\
				0 & 0 & 0
	             \end{pmatrix},
	\ee
	and the solution to \eqref{eq:diff_sys_r3_s3} starting from the point $z^0$ is given by
	\be\label{eq:zt_rank2complex}
		\left\{\begin{aligned}
			z_1(t) &= (b(\lambda)_2+z^0_1)\cos t-(z^0_2-b(\lambda)_1)\sin t-b(\lambda)_2,\\
			z_2(t) &= (z^0_2-b(\lambda)_1)\cos t+(b(\lambda)_2+z^0_1)\sin t+b(\lambda)_1,\\			
			z_3(t) &= b(\lambda)_3t+z^0_3.
		\end{aligned}\right.
	\ee
	A necessary condition for the existence of equilibria is that $\lambda\in \Lambda_6\setminus \Xi_{11}$, where
	\be
		\Xi_{11}:=\left\{\lambda\in \Lambda_6\mid b(\lambda)_3\ne 0\right\},
	\ee
	and the primitives $w$ associated with a covector $\lambda\in \Xi_{11}$ are supported within the set  $\{\mathscr{L}(\lambda,t)\mid t\in \R\}\subset \G$, where $\mathscr{L}(\lambda,t)$ is as in \eqref{eq:ricostruzione_w} provided that $z(t)$ is given by \eqref{eq:zt_rank2complex} with $z^0=0$. In particular, $\abn{\Xi_{11}}$ is a sub-analytic set of codimension at least $3$ in $\G$.
	
	If instead $\lambda\in\Lambda_6\setminus \Xi_{11}$ the set of equilibria coincides with the line
	$\left\{(-b(\lambda)_2,b(\lambda)_1,\eta)\mid \eta\in \R\right\}$, and the curves $z(t)$ in \eqref{eq:zt_rank2complex} are circles contained in the plane $\{(z_1,z_2,z_3^0)\mid (z_1,z_2)\in \R^2\}$ with center in the equilibrium $(-b(\lambda)_2,b(\lambda)_1,z_3^0)$. In particular, these curves pass through the equilibrium if and only if $z^0$ is an equilibrium itself, in which case the components $z_1(t)$ and $z_2(t)$ remain constant (instead, the concatenation allows the coordinate $z_3$ to vary within the line of equilibria). Recalling that we are interested only in concatenations of solutions to  \eqref{eq:diff_sys_r3_s3} starting from the origin, we introduce the sets
	\be\begin{aligned}
		\Xi_{12}&:=\left\{\lambda\in \Lambda_6\mid b(\lambda)_1\ne 0,\, b(\lambda)_3= 0 \right\},\\
		\Xi_{13}&:=\left\{\lambda\in \Lambda_6\mid b(\lambda)_2\ne 0,\, b(\lambda)_3= 0\right\},\\
		\Xi_{14}&:=\left\{\lambda\in \Lambda_6\mid b(\lambda)_1=b(\lambda)_2=b(\lambda)_3= 0 \right\}, 
	\end{aligned}\ee
	Our discussion shows that $\abn{\Xi_{12}}$ and $\abn{\Xi_{13}}$ are sub-analytic sets of codimension at least $3$ in $\G$. On the other hand, the singular curves $\gamma_{\dot{w}}$ associated with $\lambda\in \Xi_{14}$ are supported within $\left\{\exp\left( tX_3(\lambda) \right)\mid t\in \R\right\}$, and since $\Xi_{14}\subset \mathbb S(\mfk g_3^*)$  we conclude that $\abn{\Xi_{14}}$ is a sub-analytic set of codimension at least $4$ in $\G$.

	\begin{prop}\label{prop:40}
		For a Carnot group $\G$ of rank $3$ and step $3$, $\abn{\Lambda_6}$ is a sub-analytic set of codimension at least $3$ in $\G$.
	\end{prop}

	\subsection{Case VII\texorpdfstring{: $\Lambda=\Lambda_7$}{}} Here we have
	\be\label{eq:firstjord}
		N=\begin{pmatrix}
				0 & 1 & 0 \\
				0 & 0 & 1 \\
				0 & 0 & 0
	             \end{pmatrix}.
	\ee
The solution to \eqref{eq:diff_sys_r3_s3} starting from the point $z^0$ is given by
	\be\label{eq:zt_rank2jord3}
		\left\{\begin{aligned}
			z_1(t) &= z^0_1+(b(\lambda)_1+z^0_2)t+(b(\lambda)_2+z^0_3)\frac{t^2}{2}+b(\lambda)_3\frac{t^3}{6},\\
			z_2(t) &= z^0_2+(b(\lambda)_2+z^0_3)t+b(\lambda)_3\frac{t^2}{2},\\			
			z_3(t) &= z^0_3+b(\lambda)_3t.
		\end{aligned}\right.
	\ee
	A necessary condition for the existence of equilibria is that $\lambda\in \Lambda_7\setminus \Xi_{15}$, where	
	\be
		\lambda\in \Xi_{15}:=\left\{\lambda\in \Lambda_7\mid b(\lambda)_3\ne 0\right\},
	\ee
	and the primitives $w$ starting at the origin and associated with a covector $\lambda\in \Xi_{15}$ are supported within the set  $\{\mathscr{L}(\lambda,t)\mid t\in \R\}\subset \G$, where $\mathscr{L}(\lambda,t)$ is as in \eqref{eq:ricostruzione_w} and $z(t)$ is given by \eqref{eq:zt_rank2jord3} with $z^0=0$. In particular, $\abn{\Xi_{15}}$ is a sub-analytic set of codimension at least $3$ in $\G$.

	If instead $\lambda\in \Lambda_7\setminus \Xi_{15}$ the set of equilibria is the line $\left\{(\eta,-b(\lambda)_1,-b(\lambda)_2)\mid \eta \in \R\right\}$, and a curve $z(t)$ as in \eqref{eq:zt_rank2jord3} approaches this line if and only if $z_2^0=-b(\lambda)_1$ and $z_3^0=-b(\lambda)_2$, in which case $z(t)\equiv (z_1^0,-b(\lambda)_1,-b(\lambda)_2)$ for all times. Since we are interested only in concatenations of solutions to \eqref{eq:diff_sys_r3_s3} starting from the origin, we introduce the sets
	\be
		\begin{aligned}
			\Xi_{16}&:=\left\{\lambda\in \Lambda_7\mid b(\lambda)_3= 0,\,b(\lambda)_1\ne 0, \right\},\\
			\Xi_{17}&:=\left\{\lambda\in \Lambda_7\mid b(\lambda)_3= 0,\,b(\lambda)_2\ne 0, \right\},\\
			\Xi_{18}&:=\left\{\lambda\in \Lambda_7\mid b(\lambda)_1=b(\lambda)_2=b(\lambda)_3= 0\right\}.
		\end{aligned}
	\ee
	As in the previous subsection, we deduce that $\abn{\Xi_{16}}$ and $\abn{\Xi_{17}}$ are sub-analytic sets of codimension at least $3$ in $\G$. On the other hand, the singular curves $\gamma_{\dot{w}}$ associated with $\lambda\in \Xi_{18}$ are supported within
	$\left\{\exp\left( tX_1(\lambda) \right)\mid t\in \R\right\}$, and we easily conclude that $\abn{\Xi_{18}}$ is a sub-analytic set of codimension at least $4$ in $\G$.
	\begin{prop}\label{prop:41}
		For a Carnot group $\G$ of rank $3$ and step $3$, $\abn{\Lambda_7}$ is a sub-analytic set of codimension at least $3$ in $\G$.
	\end{prop}

	\subsection{Case VIII\texorpdfstring{: $\Lambda=\Lambda_8$}{}} Here
	\be
		N=\begin{pmatrix}
				0 & 1 & 0 \\
				0 & 0 & 0 \\
				0 & 0 & 0
	             \end{pmatrix},
	\ee
	The solution to \eqref{eq:diff_sys_r3_s3} starting from the point $z^0$ is given by
	\be\label{eq:ztjordan}
		\left\{\begin{aligned}
			z_1(t) &= z^0_1+(b(\lambda)_1+z^0_2)t+b(\lambda)_2\frac{t^2}{2},\\
			z_2(t) &= z^0_2+b(\lambda)_2t,\\			
			z_3(t) &= z^0_3+b(\lambda)_3t.
		\end{aligned}\right.
	\ee
	Let us define 
	\be
		\begin{aligned}
			\Xi_{19}&:=\left\{\lambda\in \Lambda_8\mid b(\lambda)_2\ne 0\right\},\\
			\Xi_{20}&:=\left\{\lambda\in \Lambda_8\mid b(\lambda)_3\ne 0\right\},
		\end{aligned}
	\ee
	and let us notice that if $\lambda\in \Xi_{19}\cup \Xi_{20}$ then there are no equilibria, so that 
		$\abn{\Xi_{19}}$ and $\abn{\Xi_{20}}$ are sub-analytic sets of codimension at least $3$ in $\G$. If instead $\lambda\in \Lambda_8\setminus (\Xi_{19}\cup \Xi_{20})$, the set of equilibria coincides coincides with the plane $\left\{ (\eta, -b(\lambda)_1, \theta)\mid \eta,\theta\in \R  \right\}$, and a trajectory $z(t)$ as in \eqref{eq:ztjordan} approaches this plane if and only if $z_2^0=-b(\lambda)_1$. To analyze concatenations starting at the origin, we introduce the sets
	\be
	\begin{aligned}
		& \Xi_{21}:=\left\{\lambda\in \Lambda_8\mid b(\lambda)_2=b(\lambda)_3=0,\ b(\lambda)_1\neq 0\right\},\\
		& \Xi_{22}:=\left\{\lambda\in \Lambda_8\mid b(\lambda)_1=b(\lambda)_2=b(\lambda)_3=0\right\},
	\end{aligned}
	\ee
and we notice that, for $\lambda\in\Xi_{21}$, trajectories through the origin never approach the plane of equilibria. In particular, $\abn{\Xi_{21}}$ is a sub-analytic set of codimension at least $3$ in $\G$.
	
	Observe that for every $\lambda\in \Xi_{22}$ any singular primitive $w$ 
	is in fact an absolutely continuous curve contained in the plane of equilibria. 	The singular curves $\gamma_{\dot{w}}$ starting at $e$ and associated with $\lambda\in\Xi_{22}$ are given as integral curves of differential system:
	\be\label{eq:curve_in_sottogruppi}
		\dot{\gamma}_{\dot{w}}(t)=\dot w_1(t)X_1(\lambda)(\gamma_{\dot{w}}(t))+\dot w_3(t)X_3(\lambda)(\gamma_{\dot{w}}(t)),
	\ee	
	and 	are therefore contained within the subgroup of $\G$ generated by $X_1(\lambda)$ and $X_3(\lambda)$, which has dimension at most $5$ in $\G$. We now distinguish two cases:
	\begin{itemize}
	\item If $\dim\mfk g_2\geq 2$, then $\Xi_{22}\subset\mathbb S(\mfk g_3^*)$ has codimension at least 6, therefore $\abn{\Xi_{22}}$ is a sub-analytic set of codimension at least 1 in $\G$.
	\item If $\dim\mfk g_2=1$, then we conclude by the following lemma.
	\end{itemize}
	
	\begin{lemma}\label{lem:ehmehmehm}
	Let $\G$ be a Carnot group of rank 3 and step 3 such that $\dim(\mfk g_2)=1$. Then $\G$ is isomorphic to $\mathbb H\times\R$ for some   Carnot group $\mathbb H$  of rank 2 and step 3, and $\abn{}$ is an analytic manifold of codimension 3. 
	\end{lemma}
	\begin{proof}
	The first part of the statement follows by noticing that the map $[\cdot,\cdot]:\mfk g_1\times\mfk g_1\to\mfk g_2$ can be identified with a non-zero skew-symmetric bilinear form on $\mfk g_1$, hence it has a one-dimensional kernel, say, $\Span\{X_3\}$. We then have $\G=\mathbb H\times\R$, where $\mathbb H$ is the subgroup generated by $X_1$ and $X_2$.
	
	By \cite[Proposition 2.7]{OttaVitto} we have $\abn{}=\mathrm{Abn}_{\mathbb H}\times\R$ and we distinguish two cases:
	\begin{itemize}
	\item $\mathbb H$ is the free group of rank 2 and step 3, and $\mathrm{Abn}_{\mathbb H}=\exp(\Span\{X_1,X_2\})$ by the discussion in Section~\ref{sec:CG_r2_s3}. 
	\item $\mathbb H$ is (isomorphic to) the Engel group with $[X_1,[X_1,X_2]]\neq 0$ and $[X_2,[X_1,X_2]]= 0$. It is well-known (see e.g.~\cite[Section 3]{SussCornucopia} or~\cite[p. 541]{GoleKaridi}) that $\mathrm{Abn}_{\mathbb H}=\exp(\Span\{X_2\})$.
	\end{itemize}
	In both cases the conclusion is immediate.
	\end{proof}
	
	\begin{prop}
		For a Carnot group $\G$ of rank $3$ and step $3$, $\abn{\Lambda_8}$ is a sub-analytic set of codimension at least $1$ in $\G$.
	\end{prop}

	\begin{remark}\label{rem:ex3}
	In \cite[Section 5]{LLMV_CAG} the authors provided an example of a Goh singular curve that is not better than Lipschitz continuous, as well as an example of a spiral-like Goh singular curve. We can recover both examples in the framework of the discussion of the present section (case VIII).
	
	Let $\mathbb F$ be the free Carnot group of rank 3 and step 3, and consider the curve $\gamma_{\dot w}$ as in~\eqref{eq:curve_in_sottogruppi}. Choosing $ w_1$ and $ w_3$ arbitrarily in Lip$([0,1])$ we obtain a Goh singular curve with no regularity beyond the Lipschitz one. Choosing
	\[
	(w_1(t),w_3(t))=\left(t\cos(\log(1 - \log |t|)), t\sin(\log(1 - \log |t|)) \right)
	\] 
	we recover the spiral-like example. Using Proposition~\ref{prop:singular} and computations similar to those in Remark~\ref{rem:ex2} we obtain that $X_{223}\not\in \IM(G_e^{\dot w})$, i.e., the two curves just constructed are singular and associated with the unique covector $\lambda\in\mfk g_3^*$ such that $M(\lambda)_{ij}=0$ with the exception of $M(\lambda)_{12}=\lambda_{223}=1$.
\end{remark}

	\subsection{Case IX\texorpdfstring{: $\Lambda=\Lambda_9$}{}} The condition $M(\lambda)=0$ implies that the $\mfk{g}_3^*$ component of $\lambda$ is zero, and this implies that $v(\lambda)\ne 0$ for otherwise the covector $\lambda$ itself would be zero. Solutions to \eqref{eq:diff_sys_r3_s3} are therefore lines through the origin.

	\begin{prop}
	For a Carnot group $\G$ of rank $3$ and step $3$, $\abn{\Lambda_9}$ is a sub-analytic set of codimension at least $2$ in $\G$.
	\end{prop}
	
	The proof of Theorem~\ref{thm:S_r3_s3} is complete.
	
		\subsection{Proof of Theorem~\ref{thm:S_r3_s3free}}\label{sec:proof_free} We sketch in this section how to show Theorem~\ref{thm:S_r3_s3free} if $\G$ is the free Carnot group of rank $3$ and step $3$. By Remark~\ref{rem:remark_free}, it is easy to conclude that $\abn{\Lambda_1}$ is a sub-analytic set of codimension $3$ in $\G$. A similar reasoning shows that the same conclusion holds also for $\abn{\Lambda_i}$, for $i=2,3,4$; actually, the codimensions of $\abn{\Lambda_2}$, $\abn{\Lambda_3}$, $\abn{\Lambda_4}$ are 4, 3, 4, respectively. 
		
		To see that  $\abn{\Lambda_5}$ is a sub-analytic set of codimension $3$ in $\G$ it suffices to analyze the case of $\abn{\Xi_9},\abn{\Xi_{10}}$, taking into account the two conditions imposed on $b(\lambda)$ and the further one given by $\det(M(\lambda))=0$. 
		
	The lower bound 3 on the codimension of $\abn{\Lambda_6},\abn{\Lambda_7}$ is already stated in Propositions~\ref{prop:40} and~\ref{prop:41}. However, it can be showed that they are sub-analytic sets of codimension $4$ and $5$ in $\G$, respectively. Indeed, the Jordan normal form presented in \eqref{eq:firstjord} is a condition of codimension $2$ on the $\mfk{g}^*_3$ component of $\lambda$ by, e.g., \cite[\S 5.6]{Ar_Mat}.
		
	In order to study the codimension of $\abn{\Lambda_8}$ it suffices to study $\abn{\Xi_{22}}$. Here, the trajectories of 	singular curves associated with a fixed $\lambda\in\Xi_{22}$ sweep a 5-dimensional subgroup. On the other hand, $\lambda\in\Xi_{22}$ imposes 9 independent conditions on $\lambda$ itself: 7 come from $\lambda\in\mathbb S(\mfk g_3^*)$ and 2 more are consequences of the prescribed normal form $N$ of $M(\lambda)$. Indeed, the prescribed normal form is a constraint of codimension $1$ \cite[\S 5.6]{Ar_Mat} and leads to the matrix 
	\be
		\begin{pmatrix}
				a & 1 & 0 \\
				0 & a & 0 \\
				0 & 0 & -2a
	             \end{pmatrix},\ \ a\in \R.
	\ee
	But then,  by the rank-one condition,  $a=0$ and we conclude.
	
	Finally, an easy argument shows that $\abn{\Lambda_9}=\exp(\mfk g_1)$ is an analytic manifold of codimension $11$ in $\G$, and the proof follows.

	\section{An open problem: the free Carnot group of rank 2 and step 5}\label{sec:r2s5}
	We discuss in this section the case of Carnot groups of rank $2$ and step $5$, where the dynamics of singular controls is lead by a quadratic system of differential equations.
	We derive explicitly such equations, but we leave as an open question their qualitative analysis. Let $u\in L^1([0,1],\R^2)$, $w\in AC([0,1],\R^2)$ be the primitive of $u$ and $\gamma_u$ be the singular trajectory associated with $u$. With the same conventions as in Section~\ref{sec:CGr2s4},  by Remark~\ref{rem:Goh} we can find a covector $\lambda\in \mathbb{S}(\mfk{g}_3^*\oplus \mfk{g}_4^*\oplus \mfk{g}_5^*)$ such that: 
	\be
		\lambda_{w(t)u(t)j}+\int_0^t\lambda_{w(\tau_1)u(\tau_1)u(t)j}d\tau_1+\iint_{0\le \tau_2\le \tau_1\le t}\lambda_{w(\tau_2)u(\tau_2)u(\tau_1)u(t)j}d\tau_2 d\tau_1=0
	\ee
	for a.e. $t\in [0,1]$ and $j=1,2$. The skew-symmetric matrix $\mathscr{M}_u(\lambda,t)\in M_2(\R)$ in \eqref{eq:matrix_abnormal} is given by
	\be
		\mathscr{M}_u(\lambda,t)_{ij}=\lambda_{w(t)ij}+\int_0^t\lambda_{w(\tau_1)u(\tau_1)ij}d\tau_1+\iint_{0\le \tau_2\le \tau_1\le t}\lambda_{w(\tau_2)u(\tau_2)u(\tau_1)ij}d\tau_2 d\tau_1,\ \ i,j=1,2,
	\ee
	and we have:
	\be\label{eq:pf_r2_s5_two}
		\sum_{i=1}^2u_i(t)\left(  \lambda_{i12}+\lambda_{w(t)i12}+\int_0^t\lambda_{w(\tau)u(\tau)i12}  \right)d\tau=0,\ \ \text{for a.e. }t\in [0,1].
	\ee
	We therefore conclude that $(u_1(t),u_2(t))$ is parallel to 
	\be
		\left(\lambda_{212}+\lambda_{w(t)212}+\int_0^t\lambda_{w(\tau)u(\tau)212}d\tau,-\lambda_{112}-\lambda_{w(t)112}-\int_0^t\lambda_{w(\tau)u(\tau)112}d\tau\right)\in \R^2
	\ee for a.e. $t\in [0,1]$. Let us recall the relations (compare with Definition~\ref{defi:brackets})
	\be		
		\lambda_{1212}=\lambda_{2112},\ \ \lambda_{12112}=\lambda_{(12)(112)}+\lambda_{21112},\ \ \lambda_{12212}=\lambda_{(12)(212)}+\lambda_{21212}.
	\ee
	After an integration by parts and some algebraic manipulations, we obtain the  system: 
	\be\label{eq:complicatedsyst}\begin{aligned}
		\dot{z}(t)&=\begin{pmatrix} \lambda_{212} \\ -\lambda_{112}  \end{pmatrix}+\begin{pmatrix} \lambda_{2112} & \lambda_{2212} \\ -\lambda_{1112} & -\lambda_{2112}  \end{pmatrix}z(t)+\frac{1}{2}z(t)^T\begin{pmatrix} \begin{pmatrix} \lambda_{11212} & \lambda_{21212} \\ \lambda_{21212} & \lambda_{22212} \end{pmatrix} \\  \begin{pmatrix} -\lambda_{11112} & -\lambda_{21112} \\ -\lambda_{21112} & -\lambda_{22112} \end{pmatrix}  \end{pmatrix}z(t)\\ 
		&+\begin{pmatrix} \lambda_{(12)(212)} \\ -\lambda_{(12)(112)} \end{pmatrix}\int_0^t z_1(\tau)\dot{z}_2(\tau)d\tau.
	\end{aligned}
	\ee
This integro-differential system can be differentiated in $t$ to obtain a second-order differential system in $\R^2$ or, equivalently, a first-order quadratic differential system in $\R^4$: all the primitives $w$ such that $\lambda\in \IM(G_e^{\dot{w}})^\perp$ are obtained by concatenation of (the first two components of) the integral curves of such extended system.

Let us go back to the system~\eqref{eq:complicatedsyst}, which is set in $\R^2$ with variable $z=(z_1,z_2)\in \R^2$. One can however set it in the first Heisenberg group $\mathbb H^1$ by adding a new variable $\theta=\theta(t)$: if $\mathbb H^1$ is identified with $\R^3_{z_1,z_2,\theta}$ by  exponential coordinates of the second type (see e.g.~\cite[Proposition 3.5]{LLMV_GAFA}) in such a way that a basis of left-invariant vector fields is provided by
\be
	Z_1=\partial_{z_1},\qquad Z_2=\partial_{z_2}+z_1\partial_\theta,\qquad T=[Z_1,Z_2]=\partial_\theta,
\ee
we see that~\eqref{eq:complicatedsyst} can be equivalently written as 
\be\label{eq:sistemaHeis}
	\dot p(t)=v_1(p(t))Z_1(p(t)) + v_2(p(t))Z_2(p(t)),
\ee
where $p=(z_1,z_2,\theta)=(z,\theta)$ and
\be
	\begin{pmatrix} v_1(p) \\ v_2(p)  \end{pmatrix} :=
	\begin{pmatrix} \lambda_{212} \\ -\lambda_{112}  \end{pmatrix}+\begin{pmatrix} \lambda_{2112} & \lambda_{2212} \\ -\lambda_{1112} & -\lambda_{2112}  \end{pmatrix}z+\frac{1}{2}z^T\begin{pmatrix} \begin{pmatrix} \lambda_{11212} & \lambda_{21212} \\ \lambda_{21212} & \lambda_{22212} \end{pmatrix} \\  \begin{pmatrix} -\lambda_{11112} & -\lambda_{21112} \\ -\lambda_{21112} & -\lambda_{22112} \end{pmatrix}  \end{pmatrix}z 
	+  \begin{pmatrix} \lambda_{(12)(212)} \\ -\lambda_{(12)(112)} \end{pmatrix}\theta.
\ee
The system~\eqref{eq:complicatedsyst} in $\R^2$ is then equivalent to the first-order, quadratic {\em horizontal} dynamical system~\eqref{eq:sistemaHeis} in $\mathbb H^1$. It would be interesting to know, at least, how the associated trajectories approach the equilibria. 

\bibliographystyle{abbrv}
	\bibliography{Biblio}
	
\end{document}